\documentclass[a4paper, 11pt, reqno]{amsart}
\usepackage{tikz,pgf,subcaption, enumerate}
\usepackage{amsmath,amsthm,amssymb,amsfonts, mathrsfs, indentfirst, pdflscape, url, enumitem, setspace}
                                                \usepackage[margin=1.2in, centering]{geometry}
\usepackage[plainpages=false,colorlinks,hyperindex,pdfpagemode=None,bookmarksopen,linkcolor=red,citecolor=blue,urlcolor=black]{hyperref}

\setlist[enumerate]{leftmargin=*}
\newcommand\be{\begin{equation}}
\newcommand\ee{\end{equation}}

\def\no[#1]{\| #1\|}
\def\dd[#1,#2]{\frac{#1}{#2}}
\def\peter[#1]{\langle #1\rangle}
\def\sno[#1]{\no[\mathcal{D}.#1]}

\newcommand{\Q}{\mathbb{Q}}
\newcommand{\R}{\mathbb{R}}
\newcommand{\Z}{\mathbb{Z}}

\newcommand{\G}{\mathbf{G}}

\newcommand{\SL}{\mathrm{SL}}

\newcommand{\SO}{\mathrm{SO}}

\newcommand{\lieg}{\mathfrak{g}}
\newcommand{\liea}{\mathfrak{a}}

\newcommand{\lieu}{\mathfrak{u}}

\newcommand{\supp}{\mathrm{supp}}

\newcommand{\vol}{\mathrm{vol}_H}

\newcommand{\diag}{\mathrm{diag}}

\newcommand{\e}{\epsilon}

\newcommand{\vp}{\varphi}

\newcommand{\bv}{\boldsymbol{v}}

\newcommand{\bx}{\boldsymbol{x}}
\newcommand{\by}{\boldsymbol{y}}

\newcommand{\lb}{\left\lbrace}
\newcommand{\rb}{\right\rbrace}

\newcommand{\al}{\left\langle}
\newcommand{\ar}{\right\rangle}
\newcommand{\ra}{\rightarrow}
\newcommand{\mA}{\mathcal{A}}
\newcommand{\mC}{\mathcal{C}}
\newcommand{\mE}{\mathcal{F}}
\newcommand{\dint}{\displaystyle\int}
\newcommand{\dsum}{\displaystyle\sum}
\newcommand{\ol}{\overline}
\newcommand{\mU}{\mathcal{U}}
\newcommand{\mM}{\mathcal{M}}
\newcommand{\mL}{\mathcal{L}}

\theoremstyle{plain}
\newtheorem{thm}{Theorem}
\newtheorem{cor}{Corollary}

\newtheorem{lem}{Lemma}
\newtheorem{prop}{Proposition}

\newtheorem*{remark*}{Remark}

\author{Kathryn Dabbs}
\address{Department of Mathematics \\ Univerisity of Texas \\1 University Station, Austin, TX 78712}
\email{kdabbs@math.utexas.edu}
\author{Michael Kelly}
\address{Department of Mathematics \\ Univerisity of Michigan \\530 Church St., Ann Arbor, MI 48109}
\email{michaesk@umich.edu}
\author{Han Li}
\address{Department of Mathematics and Computer Science \\ Wesleyan Univerisity \\265 Church Street, Middletown, CT 06459}
\email{hli03@wesleyan.edu}
\title[Effective Equidistribution]{Effective equidistribution of translates of \\ maximal horospherical measures in the space of lattices}
\date{\today}

\begin{document}
\vspace*{-.4cm}
\maketitle
\vspace*{-.7cm}
\begin{abstract}
	Recently Mohammadi and Salehi-Golsefidy gave necessary and sufficient conditions under which certain translates of homogeneous measures converge, and they determined the limiting measures in the cases of convergence. The class of measures they considered includes the maximal horospherical measures. In this paper we prove the corresponding effective equidistribution results in the space of unimodular lattices. We also prove the corresponding results for probability measures with absolutely continuous densities in rank two and three. Then we address the problem of determining the error terms in two counting problems also considered by Mohammadi and Salehi-Golsefidy. In the first problem, we determine an error term for counting the number of lifts of a closed horosphere from an irreducible, finite-volume quotient of the space of positive definite $n\times n$ matrices of determinant one that intersect a ball with large radius. In the second problem, we determine a logarithmic error term for the Manin conjecture of a flag variety over $\Q$.
\end{abstract}

\section{Introduction}

Several important and recurring problems in homogeneous dynamics concern the equidistribution properties of closed unipotent orbits. These problems have been studied for many years by many authors, and they are important because of their connections to geometry and number theory. It often happens that one is interested in proving not only an equidistribution result, but also a quantitative bound on the discrepancy of the equidistribution. There are two reasons for this: (1) knowledge of the rate of equidistribution sheds light on the regularity and rigidity of the dynamics; and (2) in applications, particularly in counting problems, effective rates of equidistribution play a fundamental role in determining an error term for any relevant estimates. \newline 
\indent A fundamental example of the equidistribution of closed unipotent orbits is the equidistribution of long closed horocycles in $M=\SO_{2}(\R)\backslash\SL_2(\R)/\SL_2(\Z)$, the modular surface with the Poincar\'e metric. For any $y>0$ there exists a unique closed horocycle $h_y$ in $M$ of length $\tfrac{1}{y},$ and $h_y$ equidistributes in $M$ as $y$ tends to zero. See for example \cite{Dani1978} or \cite{Zagier}. That is, if $\nu_y$ is the probability measure on $M$ that puts uniform mass on $h_y$, then $\nu_y$ converges weakly to the uniform probability measure on $M$. See Sarnak's paper \cite{Sarnak} for a generalization to general non-compact, finite-volume Riemann surfaces. From a dynamical point of view, using the fact that the horocycles $h_{y}$ are geodesic translates of any fixed closed horocycle, one can prove the equidistribution using the mixing properties of the geodesic flow. This idea originates in the thesis of G. Margulis \cite{M}.\newline
\indent The discrepancy estimates for this equidistribution problem are well studied. There are currently two main approaches to obtain such estimates. In the method of Sarnak and Zagier \cite{Sarnak, Zagier} one associates an Eisenstein series to each $\nu_{y}$ and uses the analytic continuation of the Eisenstein series (due to Selberg \cite{Selberg}) to produce an effective rate. Alternatively, using the ideas originating in Margulis's thesis, one can use the spectral gap for $M$ to achieve an effective rate for the equidstribution of the long closed horocycles. It is a well known result of Zagier \cite{Zagier} that the rate of equidistribution is $O(y^{3/4-\epsilon})$ for each $\epsilon>0$ if and only if the Riemann hypothesis is true. He also showed that the rate of equidistribution is at least $o(y^{1/2})$, which is of the same strength as the prime number theorem. \newline
\indent  The story for long closed horocycles in general rank one spaces is similar: the full horocycle is expanded or contracted depending on the direction it is translated. In higher rank, however, the closed horospheres (the closed maximal unipotent orbits) can be simultaneously contracted and expanded as they are translated along a given geodesic. This complication has caught the attention of many mathematicians over the years, and there are many special circumstances for which we know how to control it. For instance, when investigating certain problems concerning Diophantine approximation with weights, Kleinbock and Weiss \cite{KW} proved an equidistribution theorem for the translates of {\it minimal} horospherical measures\footnote{ A horospherical subgroup is the unipotent radical of a proper parabolic subgroup. A horospherical subgroup is {\it minimal} if it is the unipotent radical of a {\it maximal} parabolic subgroup, and it is {\it maximal} if it is the unipotent radical of a {\it minimal} parabolic subgroup.} in the presence of simultaneous expansion and contraction. An effective form of their result was obtained in \cite{KM} and was generalized in the recent paper \cite{KSW}.  \newline 
	\indent Another situation in which we know how to control simultaneous expansion and contraction was the subject of the recent work of Mohammadi and Salehi-Golsefidy \cite{MS}. They provide necessary and sufficient conditions under which translates of certain homogeneous measures (including the maximal horospherical measures) converge, and they determine the limiting measures in the cases of convergence. Similar results can be found in an earlier work of Shah and Weiss \cite{SW}, in which a similar collection of translates is considered. To reiterate: the difficulty one encounters in the higher rank setting is that a closed horospherical orbits can both expand and contract while being translated in a particular direction. This phenomenon makes it difficult to determine the convergence of translates. Once more, it makes it difficult to achieve effective rates of convergence. \newline
	\indent It is our objective in this paper to establish the rates of convergence for the main results in \cite{MS} in the space of unimodular lattices. Our first theorem establishes effective equidistribution for translates of maximal horospherical measures. This result is an analogue of a similar result for minimal horospherical measures originally obtained in an ineffective form in \cite{KW} and effective form in \cite{KM}. Our method of proof closely mirrors that of Kleinbock-Margulis \cite{KM} for the minimal horospherical case.
	
\subsection{Statement of Results} Let $n>2$ be an integer, $G=\SL_{n}(\R)$, $\Gamma=\SL_{n}(\Z)$, and $A$ be the subgroup of $G$ consisting of positive diagonal matrices. For an element $a\in A$ we will use the notation 
	\[
		a=\diag(a_{1},...,a_{n}).
	\]
Let $\Delta=\{ \alpha_{1},...,\alpha_{n-1}\}$ be simple roots of $G$ with respect to $A$ given by 
	\[
		\alpha_{i}(a)= \dfrac{a_{i}}{a_{i+1}},
	\]
 and let $ \lambda_{\alpha_{1}},...,\lambda_{\alpha_{n-1}}$ be the corresponding fundamental weights
 	\[
		\lambda_{\alpha_{i}}(a)=a_{1}\cdots a_{i}.
	\]

For each $E\subset \Delta$, let $P_{E}$ be the associated standard parabolic subgroup (see \cite{Hum}). For example, $P_{\Delta}=G$ and $P_{\emptyset}$ is the group of upper triangular matrices in $G$. Let $Q_{E}$ be the group generated by the one parameter unipotent subgroups of $P_{E}$. The group $Q_{\emptyset}$ is a maximal unipotent subgroup of $G$, and for it we reserve the special notation $U=Q_{\emptyset}$. Finally, we let $\mu_{E}$ be the unique invariant probability measure supported on $Q_{E}\Gamma$ in $G/\Gamma$. We use the notation $\mu=\mu_{\emptyset}$, and set $m=\mu_{\Delta}$ for the $G$-invariant probability measure on $G/\Gamma$.  We will now state a special case\footnote{In \cite{MS} this theorem is proved in much greater generality, e.g. $G$ does not necessarily have to be $\Q$-split.} of the main result of \cite{MS}.


\begin{thm}[{\cite[Theorem 1]{MS}}] \label{MS}
Let $\{a_{k}\}\subset A$ and $E\subset \Delta$. Then 
\begin{enumerate}
	\item If $\lambda_\alpha(a_{k})\ra0$ as $k\ra\infty$ for some $\alpha\not\in E$, then $a_{k}\mu_{E}$ diverges in the space of Borel probability measures on $G/\Gamma$. 
	\item Let $E\subset F\subset \Delta$. If $\lambda_{\alpha}(a_{k})=1$ for any $\alpha\not\in F$ and $\lambda_{\alpha}(a_{k})\ra\infty$ as $k\ra\infty$ for any $\alpha\in F\backslash E$, then $a_{k}\mu_{E}$ converges to $\mu_{F}$ as $k\ra\infty$.
\end{enumerate}
\end{thm}

\indent \autoref{MS} can be thought of as identifying ``cones'' in $A$ that govern the convergence of the translates of the measures $\mu_{E}$. For each $E\subset \Delta$, let
$$
\mathcal{C}_{E}=\lb a\in A : \lambda_{\alpha}(a)>1 \text{ for each } \alpha\in \Delta \setminus E\rb.
$$
If $\{a_k\}$ tends to infinity {\it away from the boundary} in $\mC_{E}$ (a notion made precise by the above theorem), then $a_k.\mu$ tends to $\mu_{E}$. If $E=\emptyset$, then we call the set $\mC=\mC_{\emptyset}$  the convergence cone. Each of these cones contains the cone 
$$
\mathcal{A}=\lb a\in A : \alpha(a)>1 \text{ for each }\alpha\in\Delta\rb
$$
which we call the {\it positive} or {\it fundamental} Weyl chamber. 

Our main result is an effective version of \autoref{MS} for the translates of the maximal horospherical measure $\mu$.
\begin{thm}\label{main1}
There exists a constant $\delta=\delta(n)>0$ such that for any $\vp\in C^{\infty}_{comp}(G/\Gamma)$ there exists a constant $C=C(\vp,n)>0$ such that  for any  $a\in A$ 
\begin{equation}\label{equiTheoremMain}
\Big|\int_{G/\Gamma}\vp(a.z)d\mu(z)-\int_{G/\Gamma}\vp(z) dm(z)\Big| < C\left(\min_{\alpha\in \Delta}\lambda_{\alpha}(a)\right)^{-\delta}.
\end{equation}
\end{thm}
We remark the above theorem is trivial when $a\not\in \mC$. To see this suppose that $a\not\in \mC$ and observe (by the definition of $\mC$) for any $\delta>0$, $\left(\min_{\alpha\in \Delta}\lambda_{\alpha}(a)\right)^{-\delta}\geq1$ and by taking $C=2\sup|\vp|$ we find the inequality is always satisfied. Our next result is a generalization of \autoref{main1}, and is an effective version of \autoref{MS}. After a suitable decomposition of measures, its proof proceeds by repeatedly applying \autoref{main1} to certain marginals of $\mu_{E}$.
\begin{thm}\label{main2}
	Let $E\subset F\subset \Delta$. There exists a constant $\delta=\delta(n)>0$ such that for any $\vp\in C^{\infty}_{comp}(G/\Gamma)$ there exists a constant $C=C(\vp,n)>0$ such that  for any  $a\in A$  we have 
		\[
			\Big|\int_{G/\Gamma}\vp(a.h)d\mu_{E}(h)-\int_{G/\Gamma}\vp(h) d\mu_{F}(h)\Big| < C\left(  \min_{\alpha\in F\backslash E}\lambda_{\alpha}(a) \right)^{-\delta}.
		\]
\end{thm}

Our next result is an effective version of \autoref{MS} for absolutely continuous measures. We are able to obtain effective results in the full convergence cone $\mC$ when $n=3$ and $n=4$. For $n>4$ we are able to prove an effective result for flows in a cone that is strictly larger than $\mA$. However, in general, we are unable to handle the absolutely continuous case for the full convergence cone.  
\indent For each $j=1,...,n-1$ we define 
\[
	\mathcal{C}_{j}=\lb   \diag(e^{r_{1}},...,e^{r_{n}})\in A ~:~ \min\lb r_{i} : i=1,...,j \rb \geq \max\lb r_{s} : s=j+1,...,n \rb   \rb, 
\]
and $\tilde{\mC}=\mC_{1}\cup \cdots \cup \mC_{n-1}$.  Let $\nu_{U}$ be the Haar measure on $U$ that is equal to $\mu$ on a fundamental domain of $\Gamma$ in $G/\Gamma$. Now we can state our result regarding absolutely continuous densities. 

\begin{thm}\label{main3}
There exists a constant $\delta=\delta_{n}>0$ such that for any compact subset $L$ of $G/\Gamma$, for any $f\in C^{\infty}_{comp}(U)$, $\varphi\in C^{\infty}_{comp}(G/\Gamma)$, there exists a constant $C=C(f, \vp, L,n)>0$ such that for any $z\in L$ and $a \in \tilde{\mC}$ we have
\be
\Big|\int_U f(u)\vp(auz) d\nu_U(u)-\int_Uf\cdot \int_{G/\Gamma} \vp\Big|<C\cdot \left( \min_{\alpha\in \Delta} \lambda_{\alpha}(a) \right)^{-\delta}.
\ee
\end{thm}

We remark that \autoref{main3} can be obtained from Theorem 1.3 of Kleinbock-Margulis \cite{KM} in conjunction with Fubini's theorem. This was pointed out to us by Kleinbock. The key point is that if $a\in\mC$ completely expands a minimal horospherical marginal of $\nu_{U}$, then the equidistribution of that particular marginal will force the equidistribution of $fd\nu_{U}$. Theorem 1.3 of Kleinbock-Margulis \cite{KM} exactly describes the equidistribution of minimal horospherical measures. \newline
 \indent Notice that $\tilde{\mC}\not\subset \mC$. For instance, $a=\diag(e^{1},e^{-2},e^{1})\in\mC_{1}\subset \tilde{\mC}$ but $a\not\in\mC$. So \autoref{main3} is only non-trivial for $a\in\tilde{\mC}\cap \mC$. See the remark after the statement of \autoref{main1}. In \cite{MS} it was pointed out by Mohammadi and Salehi-Golsefidy  that when $n=5$
 \[
	a_{0}=\diag(e^{6},e^{7},e^{-12},e^{9},e^{10}) 
 \]
 is an element of $\mC$, but it does not fully expand a minimal horospherical subgroup. It follows that $a_{0}\not\in\tilde{\mC}$ (this can be shown directly) and, consequently, that $\tilde{\mC}$ is not even convex when $n=5$. From here it is an easy exercise to show that $\mC\subset \tilde{\mC}$ if and only if $n=3$ or 4. Thus we have the following corollary of \autoref{main3}.

\begin{cor}\label{main4}
Suppose $n=3$ or 4. There exists a constant $\delta=\delta_{n}>0$ such that for any compact subset $L$ of $G/\Gamma$, for any $f\in C^{\infty}_{comp}(U)$, $\varphi\in C^{\infty}_{comp}(G/\Gamma)$, there exists a constant $C=C(f, \vp, L,n)>0$ such that for any $z\in L$ and $ a\in A$ we have
\be
\Big|\int_U f(u)\vp(a. uz) d\nu_U(u)-\int_Uf\cdot \int_{G/\Gamma} \vp\Big|<C\cdot \left( \min_{\alpha\in \Delta} \lambda_{\alpha}(a) \right)^{-\delta}.
\ee
\end{cor}

\subsection{Applications}
In our first application we consider a geometric counting problem first considered in \cite{EM}. Let $K=\SO_{n}(\R)\leq G$ and $X=K\backslash G$ be the corresponding Riemannian symmetric space arising from $G$. If $U$ is a maximal unipotent subgroup of $G$, then  $\mathcal{U}=K\backslash Kg U$ is a horosphere in $X$ and all horospheres in $X$ can be realized in this way. We let $\Xi$ be the space of horospheres in $X$.  Let $\mM=X/\Gamma$ and let $\pi:X\ra \mM$ be the covering map.  Suppose that $\mathcal{U}$ is a horosphere in $X$ such that $\ol{\mU}= \pi(\mU)$ is closed in $\mM$. We are interested in estimating how many lifts of $\ol{\mU}$ intersect a given ball $B(x,R)$ in $X$. That is, we wish to analyze the asymptotic behavior of the quantity 
	\begin{equation}\label{geometricQuantity}
		\# \lb \mU\in \Xi ~:~ \pi(\mU)=\ol{\mU}\text{ and }\mU\cap B(x,R)\neq \emptyset \rb.
	\end{equation}
In the rank one case ($n=2$) it was shown by Eskin and McMullen \cite{EM} that the quantity in \autoref{geometricQuantity} is asymptotic to the volume of $B(x,R)$ (times a suitable constant). The analogous result for higher rank ($n>2$) was established by Mohammadi and Salehi-Golsefidy \cite{MS}. Our first theorem is an effective form of this result for $G=\SL_{n}(\R)$. In principle, the Eskin-McMullen example can be made effective using Sarnak's effective equidistriubtion of low-lying horocycles \cite{Sarnak}. We prove here, as far as we know, the first effective result for this counting problem in higher rank.

\begin{thm}\label{horo1thm}
	Let $\ol{\mU}$ be a closed horosphere in $\mM$ and $x_{0}\in X$ be the identity coset. Then there is a constant $C>0$, depending only on the dimension, and $\delta>0$ such that
		\[
			\begin{split}
			\# \lb \mU\in \Xi : \pi(\mU)=\ol{\mU}\text{ and }\mU\cap B(x_{0},R)\neq \emptyset \rb = C & \dfrac{\mathrm{vol}(\ol{\mU})}{\mathrm{vol}(\mM)}  \mathrm{vol}(B(x_{0},R)) \\ &+O\left(\mathrm{vol}(B(x_{0},R))^{1-\delta}\right).
			\end{split}
		\]
\end{thm}

To prove the above theorem we only need to use the effective equidistribution for directions coming from the interior of $\mA$. Consequently, our proof of \autoref{horo1thm} can be adapted to prove \cite[Theorem 3]{MS} using only the wavefront lemma of Eskin-McMullen \cite{EM}.\newline
	\indent For our second application we consider the Manin conjecture for flag varieties over $\Q$. This problem was solved for generalized flag varieties by Franke, Manin, and Tschinkel in \cite{FMT}. Their proof uses Langland's analytic continuation of higher rank Eisenstein series, and the method of that paper produces what is essentially the best possible error term. Here we will prove an effective form of their theorem that produces an inferior error term, but by using our effective equidistribution results in place of Eisenstein series. A dynamical proof of a more general result\footnote{In \cite{MS} Mohammadi and Salehi-Golsefidy are also able to handle the counting for heights with respect to arbitrary metrized line bundles.} is provided in \cite{MS}, and it is this proof that we effectivize. 
	Consider the standard representation of $G$ on $\R^n$. It is well known that the stabilizer of any flag in $\R^n$ is a parabolic subgroup of $G$. Conversely, any parabolic subgroup of $G$ stabilizes a flag in $\R^n$. It then follows that any flag variety over $\Q$ can be realized as $X=G/P$ for some parabolic subgroup $P$ of $G$. The anticanonical line bundle of $X$ is induced by a character $\rho_{P}$ of $P$ by $\mL=G\times \R/\sim$ where $(g,x)\sim (gp,\rho_{P}(p)x)$. It follows from \cite[Section 12]{BT65} that $\rho_{P}$ is the highest weight of a unique irreducible representation $\eta:G\ra GL(V)$ which is strongly rational over $\Q$, there is a $v_{0}\in V(\Q)$ such that 
		\[
			P=\lb  g\in G~:~ \eta(g)[v_{0}]=[v_{0}]  \rb
		\]
where $[v_{0}]$ is the point corresponding to $v_{0}$ in $\mathbb{P}(V)$, and $X$ is homeomorphic to $\eta(G)[v_{0}]\subset \mathbb{P}(V)$. Our counting will take place in this orbit and we henceforth identify $X$ with $\eta(G)[v_{0}]$.
We define a function $H:\mathbb{P}(V)(\Q)\ra \R^+$ by $H([v])=\| v \|$ where $v$ is a primitive integral point corresponding to the point $[v]$ and $\|\cdot\|$ is the Eucliean norm on $V$. Using $H(\cdot)$ we define the (anticanonical) height $h:X\ra \R^+$ on $X$ by 
	\begin{equation}
		h(\eta(g)[v_{0}])=H(\eta(g)[v_{0}]).
	\end{equation}
We are interested in the asymptotic behavior of 
	\[
		N(T)=\# \lb x\in X(\Q) ~:~ h(x)\leq T \rb.
	\]
In \cite[Theorem 5]{FMT} it was proven that there exists a polynomial $p$ of degree $rk(\mathrm{Pic}(X))$, such that  
	\[
		N(T)=Tp(\log(T))+o(T)
	\]
as $T\ra\infty$. It is not difficult to show that their method shows that the error term $o(T)$ can be replaced with $O(T^{1-\epsilon})$ for some $\epsilon>0$. We are able to prove the following.

\begin{thm}\label{maninThm}
	Let $X$ and $h$ be as above. Then there exists a constant $\delta>0$, and a polynomial $p(t)$ of degree $k=rk(\mathrm{Pic}(X))$ such that 
		\begin{equation}
			\# \lb x\in X(\Q) ~:~ h(x)\leq T \rb = Tp(\log(T))\left( 1+o(\log(T)^{-\delta})\right)
		\end{equation}
	as $T\ra\infty.$
\end{thm}
In the proofs of the previous two theorems we use a well developed counting technique which is due to Duke-Rudnick-Sarnak \cite{DRS} and that has been employed by a number of authors. We recommend the survey \cite{Oh} of Hee Oh for an overview of the method as well as its various applications. \newline

\subsection{Further Remarks and References}
	\indent After the initial submission of this paper we learned of a recent preprint of Shi \cite{Shi} that generalizes the main results of \cite{KSW} and includes a generalization \cite[Theorem 1.5]{Shi} of our \autoref{main3}. In both \cite[Theorem 1.5]{Shi} and our \autoref{main3} above it is required that the translates of the maximal horospherical measure contain a minimal horospherical marginal which is completely expanded. In rank four and greater, this is not always possible (see the example in \cite[\S 2]{MS} which is mentioned in the remarks following \autoref{main3} above). The proofs of both \autoref{main3} and \cite[Theorem 1.5]{Shi} follow the approach of Kleinbock-Margulis \cite{KM} which is summarized in \S\ref{overview}. It will be apparent from the remarks in \S\ref{overview} that Shi's proofs can be modified to prove a  generalization of our \autoref{main1} when $G$ is a higher rank semisimple Lie group without compact factors. \newline 
	\indent Currently it seems that new ideas are needed to prove a full generalization of \cite[Theorem 2]{MS} or even \autoref{main4} in rank greater than three. See also the remarks below the statement of Theorem 1.4 in \cite{Shi}. In \cite{MS} more general ineffective versions of the above theorems were proved, and it would be desirable to treat their effectivization for applications. In particular, it would be interesting to prove an effective version of \cite[Theorem 1]{MS} (i.e. a generalization of our \autoref{main2}), and to treat the Manin problem for a generalized flag variety with respect to an arbitrary metrized line bundle. We plan to revisit these questions in a follow-up paper. The purpose of this paper is to report this progress in a concrete setting: the space of unimodular lattices. \newline
\subsection*{Organization of the paper}
We begin by proving our effective equidistribution theorems in \autoref{translates}. In \autoref{recall-km} and \autoref{kathryn} we recall some results we will need from \cite{KM,KW} regarding Margulis's thickening technique and establishing a quantitative recurrence result for translates of maximal unipotent orbits (see \autoref{nondiv_1} and \autoref{nondiv2}). Then we finish \autoref{translates} with the proofs of \autoref{main1}, \autoref{main2}, and \autoref{main3} in \autoref{pf1}, \autoref{pf2}, and \autoref{proof_of_main3} respectively. \newline 
\indent In \autoref{horospheres} we prove \autoref{horo1thm} and then prove \autoref{maninThm} in \autoref{rational}.
\section{Translates of horospherical measures}\label{translates}

While we have stated our main results in terms of the multiplicative form of $\Delta$, we will find it convenient to prove our results in additive form. That is, we take logs, and instead of considering elements in $A$ we consider elements in its Lie algebra $\mathfrak{a}$, the vector space of traceless diagonal matrices. More specifically, for any $a \in A$, we may write $a = \exp(\diag(t_1, \dots, t_n))$, where $t_1, \dots, t_n \in \R$. Then, abusing the notation, we let $\Delta=\{ \alpha_{1}, \dots, \alpha_{n-1}  \}$, where
	\[
		\alpha_i(\diag(t_1, \dots, t_n))=t_{i}-t_{i+1}.
	\]	
The set $\Delta$ is a standard choice of simple roots of $\mathfrak{g}=\mathfrak{sl}_{n}(\R)$. The corresponding fundamental weights are given by
	\[
		\lambda_{\alpha_{i}}(\diag(t_1, \dots, t_n))=t_1 + \cdots + t_i.
	\]
Then the cones $\mA$ and $\mC_{E}$ may be identified with their logarithms as follows:
$$
\mathcal{A}=\{X \in \liea : \alpha(X) > 0 \quad \text{for each} \quad \alpha \in \Delta\},
$$
and for each $E \subset \Delta$
$$
\mathcal{C}_{E}=\{X \in \liea : \lambda_{\alpha}(X) > 0 \quad \text{for each} \quad \alpha \in \Delta \setminus E \}.
$$
%
%
%
\subsection{An overview of the method} \label{overview}
Our goal in \autoref{translates} is to prove the effective equidistribution results in \autoref{main1} and \autoref{main3}. The proofs of the two theorems use similar ideas. We will provide an overview of these ideas for \autoref{main1} and then we will comment on the additional complications that must be dealt with in the proof of \autoref{main3}.

Let $z_{0}=e\Gamma\in G/\Gamma$ be the identity coset, $a=g_t=\exp(t\boldsymbol{\theta})$, where $t>0$, and $\boldsymbol{\theta}\in\mC$ is on the unit sphere of $\mathfrak{a}$. We assume, as we may, that the test function $\vp$ in \autoref{main1} satisfies $\int_{G/\Gamma}\vp dm=0$. Let $\xi$ be a smooth function supported in $B_U(r)$ with $\int_U \xi=1$.
Then plainly
$$\int_{U.z_0} \vp(g_t z) d\mu(z)
=\int_{U}\int_{U.z_0} \xi(u)\vp(g_t z) d\mu(z) d\nu_U(u).
$$
As $g_t$ lies in the interior of the convergence cone $\mathcal{C}$, we can
write $g_t=a_tb_t$ where $a_t$ is a perturbation lying in the interior of the positive Weyl chamber and $b_t$ still lies in the interior of the convergence cone $\mathcal{C}$.
Since $b_{-t}ub_t\in U$ and the measure $\mu$ on $U.z_0$ is left invariant ($z\mapsto b_{-t}ub_tz$), we have
$$
\int_{U.z_0} \vp(g_t z) d\mu(z)=\int_{U.z_0} \vp(a_tb_t z) d\mu(z) =\int_{U.z_0}\int_{U} \xi(u)\vp(a_tub_t z)  d\nu_U(u)d\mu(z).
$$
Now we are in a position to estimate the above integral. To accomplish this we write $U.z_{0}$ as $U.z_{0}=B_1\cup B_2$, where $
		B_1:=\lb  z\in U.z_{0} ~:~ b_{t}\cdot z\not\in K  \rb
	$ consists of those $z$ not returning to a properly chosen large compact subset $K$ of $G/\Gamma$; and write
	\be\label{small}
		\int_{U.z_0}\int_{U}\xi(u)\vp(a_tub_t z)  d\nu_U(u)d\mu(z)=I + II
	\ee
where

	\[
		I:=\int_{B_1}\int_{U}\xi(u)\vp(a_tub_t z)  d\nu_U(u)d\mu(z)
	\]
and
	\[
		II:= \int_{B_2}\int_{U} \xi(u)\vp(a_tub_t z)  d\nu_U(u)d\mu(z).
	\]
	
	To prove \autoref{main1} it suffices to show that the integrals $I$ and $II$ in (\ref{small}) are both small. 
	To show that integral $I$ is small, we will prove in \autoref{kathryn} that the measure of $B_1$ is small. In other words, most of the points of $U.z_0$, translated by $b_t$, will return to the compact set $K$. As we will see, the return is guaranteed by the fact that $b_t$ lies in the interior of the cone $\mathcal{C}$. 
		To show that integral $II$ is small, we will use a result of Kleinbock-Margulis \cite{KM} on the effective equidistribution of the full expanding horospherical orbits. Their result will be recalled in \autoref{recall-km}.  \newline
		\indent The proof of \autoref{main3} is quite similar to the proof just outlined but there is a crucial difference. Following the discussion above, but replacing $d\mu(z)$ by $f(z)d\mu(z)$, we come to a situation where we choose sets $B_{1}$ and $B_{2}$ (which now depend on the choice of $f$) and estimate the integrals $I$ and $II$. It turns out that estimating $I$ is manageable. However the estimate of $II$ is based on effective equidistribution of the full expanding horospherical orbits of \cite{KM} (which is \autoref{equi} below). After applying this result the Sobolev norm of $h\in H \mapsto f(b_{-t}hb_tz)$ makes an appearance where $H$ is the horospherical subgroup appearing in \autoref{equi} below. We control this Sobolev norm by choosing $b_{t}$ so that $H$ is completely expanded by conjugation with $b_{t}$. But it is not always possible to choose $b_{t}\in \mC$ in this way while also choosing $a_{t}$ to lie in $\mA$. (\S2 of \cite{MS} provides an example of such a flow $g_{t}$. See the remarks following \autoref{main3} above.) This is why we are not presently able to prove \autoref{main3} for all $a\in\mC$. So the crucial difference in the proofs of \autoref{main1} and \autoref{main3}: when $f$ is a constant function (as in \autoref{main1}) there is no need to control its Sobolev norm! Without the need to control the Sobolev norm, the proof of \autoref{main1} goes through without restricting the factorization $g_{t}=a_{t}b_{t}$.

\subsection{Effective equidistribution of expanding horospheres}\label{recall-km}

Fix a right-invariant metric `$\mathrm{dist}$' on G which gives rise to the corresponding metric on $SO(n)\backslash G$. The following result is essentially \cite[Theorem 2.3]{KM}.

\begin{prop}\label{equi}  
Let $\{a_t: t>0\}$ be a diagonal flow in $G$ and $H$ the full expanding horospherical subgroup of $\{a_t: t>0\}$.
Let $z\in G/\Gamma$, $f\in C_{comp}^\infty(H)$, and $0<r<1$ be such that the map $g\mapsto g.z$ is injective on $B_G(2r)\supp(f)\subset G$. Then for any $t>0$ and any smooth function $\vp$ on $G/\Gamma$ with $\int_{G/\Gamma}\vp=0$ one has that
\be\label{decay1}
\Big|\int_H f(h)\vp(a_thz)d\nu_H(h)\Big|\ll r\cdot\no[\vp]_{\rm Lip}\cdot\int_H|f|+ r^{-k}\cdot\no[f]_\ell \cdot \no[\vp]_\ell \cdot e^{-\gamma \mathrm{dist}(a_{t},e)}
\ee
where $\gamma>0$ is an absolute constant and $k,\ell\in\Z^+$, where $k>2\ell$  depends on $\ell$ and $\dim(H)$, and the implied constant is absolute.
\end{prop}
We shall not reproduce the proof of the above proposition since it is nearly identical to the proof of  \cite[Theorem 2.3]{KM}. They prove the above proposition for the special case $a_{t}=\diag(e^{t/m},...,e^{t/m},e^{-t/n},...,e^{-t/n})$, but the general case follows easily.
\subsection{Quantitative non-divergence of unipotent flows}\label{kathryn}
For $\varepsilon > 0$ define 
$$
K_\varepsilon := \pi(\{g \in G : ||g\bv|| \geq \varepsilon \quad \text{for all} \quad \bv \in \Z^{n} \setminus \{0\}\}).
$$

\noindent In other words, $K_\varepsilon$ consists of unimodular lattices in $\R^{n}$ whose first minimum is at least $\varepsilon$. By Mahler's compactness criterion, $K_{\varepsilon}$ is a compact subset of $G/\Gamma$. 
Kleinbock and Margulis proved in \cite{KM2} that certain polynomial maps cannot escape $K_\varepsilon$ except on a set of small measure. See Theorem~5.2 from \cite{KM2}. This result was generalized in \cite{BKM}. The following Theorem from \cite{KM} is a special case of Theorem~6.2 from \cite{BKM}.

\begin{thm}[{\cite[Theorem 3.1]{KM}}]\label{KMnondiv_thm}
Let $\phi: \R^{d} \to GL_{n}(\R)$ be a map such that all coordinates are polynomial of degree not greater than $l$, and let $B$ be a ball in $\R^d$ such that for any $k=1, \dots, n-1$ and any $\bv \in \bigwedge^{k}(\Z^{n}) \backslash \{0\}$, $||\phi(x) \bv|| \geq 1$ for some $x \in B$. Then for any positive $\varepsilon \leq 1$,
$$
\lambda( \{ x \in B : \pi(\phi(x)) \notin K_{\varepsilon} \}) \ll \varepsilon^{\frac{1}{dl}}\lambda(B),
$$
where $\lambda$ is the Lebesgue measure on $\R^{d}$ and $||\cdot||$ is the Euclidean norm. 
\end{thm}
Let $d := \frac{n^2-n}{2}$, and let $\{X_1, \dots, X_d\}$ be a basis for the Lie algebra $\lieu$ of $U$. Define
$$
\Theta: \R^d \to G \quad \text{by} \quad \Theta(s_1, \dots, s_d) = \exp(s_1X_1) \dots \exp(s_dX_d).
$$
Let $g_{\mathbf{t}} := \diag(e^{t_1}, \cdots, e^{t_n})$, and define $T_{\min} := \min \limits_{1 \leq j < n} t_1 + \dots + t_j$.
We will apply \autoref{KMnondiv_thm} with $\phi: \R^d \to G$ defined by
$$
\phi: s \mapsto g_{\mathbf{t}}\Theta(s)g
$$
for a fixed $g_{\mathbf{t}} \in \mathcal{C}$ and $g \in G$. It is easy to see that this choice of $\phi$ satisfies the first condition of \autoref{KMnondiv_thm}. We will use the next proposition to show that $\phi$ satisfies the second condition.

\begin{prop}\label{nondiv_prop}
Let $\rho: G \to GL(V)$ be a representation on a finite-dimensional vector space $V$ with no nonzero $G$-invariant vectors. Then there exist $\alpha >0$ and $c_{1} > 0$ such that for any $\bv \in V$ and $g_{\mathbf{t}} \in \mathcal{C}$,
$$ 
\sup \limits_{u \in B_U(r)} ||\rho(g_{\mathbf{t}} u)\bv|| \geq c_{1}e^{\alpha T_{\min}}||\bv||, 
$$
where $c_{1}$ depends on $r$, the representation, and choice of norm, and $\alpha$ depends on the representation.
\end{prop}

Before we can prove \autoref{nondiv_prop}, we need to prove the following representation-theoretic lemma. The proof is similar to \cite{KW}, but here we need to consider more general diagonal elements: any $g_{\mathbf{t}} \in \mathcal{C}$.

\begin{lem}\label{nondiv_lem}
Let $(\rho,V)$ be a representation as in \autoref{nondiv_prop}, and define 
$$
V^{U} = \{\bv \in V : u\bv = \bv \quad \text{for all} \quad u \in U\}.
$$
Then there exist $\alpha>0$ and $c_{0} > 0$ such that for any $\bv \in V^{U}$ and $g_{\mathbf{t}} \in \mathcal{C}$,
$$
||\rho(g_{\mathbf{t}}) \bv|| \geq c_{0}e^{\alpha T_{\min}} ||\bv||,
$$
where $c_0$ depends on the choice of norm and $\alpha$ depends on the representation.
\end{lem}

\begin{proof}
Let $A$ be the subgroup of positive diagonal matrices in $G$. Let $\liea$ and $\mathfrak{u}$ be the Lie algebras of $A$ and $U$ respectively. Note that $A$ normalizes $U$, so $V^{U}$ is a $\rho(A)$-invariant subspace. Then we  can define the $\rho(A)$-equivariant projection $p: V \to V^{U}$, and we can write $V^{U}= \bigoplus\limits_{\chi \in \Psi} V_{\chi}$, where $\Psi$ is a finite set of weights and 
	\[
		V_\chi = \{ \bv \in V : \rho(\exp X)v = e^{\chi(X)}v \quad \text{for all} \quad X \in \liea\}.
	\]

Let $E_{i,j}$ be the $n \times n$ matrix with $1$ in the $ij^{th}$ entry and $0$ otherwise, and define $F_{i,j} := E_{i,i} - E_{j,j}$. For $i = 1, \dots, n-1$, define $G(i)$ to be the Lie subgroup of G whose Lie algebra is $\lieg(i) := \langle E_{i,i+1}, E_{i+1,i}, F_{i,i+1}\rangle$. Note that each $G(i)$ is a copy of $\SL_2(\R)$ in G. Also note that $\{\lieg(i) : 1 \leq i < n \}$ generates  $\lieg$, the Lie algebra of $G$. Thus $\{G(i) : 1 \leq i < n\}$ generates $G$.

Every vector in $V_{\chi}$ is fixed by $\rho(u)$ for every $u \in U$, so in particular it is fixed by $\rho(\exp E_{i,i+1})$. Then by the representation theory of $\SL_2(\R)$, $\chi(F_{i,i+1}) = m-1$, where $m$ is the dimension of the representation. Note that $\chi(F_{i,i+1}) = 0$ if and only if $\rho$ is the trivial representation of $G(i)$ on $V_{\chi}$. Since $V$ contains no nonzero vectors fixed by $G$ and $G$ is generated by $\{G(i) : 1 \leq i < n\}$, there is some $i$ such that $\chi(F_{i,i+1}) > 0$; call it $i_{0}$. Then
\begin{eqnarray*}
	\chi(\diag(t_1, \dots, t_n)) 
	&=& \chi(t_1F_{1,2}+(t_1+t_2)F_{2,3} + \dots + (t_1+\dots+t_{n-1})F_{n,n-1}) \\
	&\geq& T_{\min} \chi(F_{i_{0}, i_{0}+1}).
\end{eqnarray*}
Thus for any $g_{\mathbf{t}} \in \mathcal{C}$ and any $\bv \in V_{\chi}$, $||\rho(g_{\mathbf{t}})\bv|| \geq e^{\chi(F_{i_{0}, i_{0}+1})T_{\min}}||\bv|| := e^{\alpha_{0}T_{\min}} ||\bv||$ where $\alpha_{0} >0$.

Without loss of generality, we may assume that $||\cdot||$ is the sup norm with respect to a basis of $\rho(A)$-eigenvectors. Then, for any $g_{\mathbf{t}} \in \mathcal{C}$ and any $\bv \in V^{U}$, 
$$
||\rho(g_{\mathbf{t}})\bv|| \geq c_{0}e^{\alpha T_{\min}} ||\bv||.
$$
\end{proof}

Now, combining \autoref{nondiv_lem} with \cite[Lemma 5.1]{shah1996limit}, we can prove the proposition.

\begin{proof}[Proof of \autoref{nondiv_prop}]
Let $p: V \to V^{U}$ be as in the proof of \autoref{nondiv_lem}. Now 
	\begin{align*}
	\sup \limits_{u \in B_U(r)} ||\rho(g_{\mathbf{t}} u)\bv|| 
	&\geq \sup \limits_{u \in B_U(r)} || p(\rho(g_{\mathbf{t}} u)\bv)|| & \\
 	&= \sup \limits_{u \in B_U(r)} || \rho(g_{\mathbf{t}}) p(\rho(u)\bv))|| & \\
	&\geq c_{0} e^{\alpha T_{\min}} \sup \limits_{u \in B_U(r)} || p(\rho(u)\bv)|| &\text{by \autoref{nondiv_lem} }\\
	&\geq c_{1}e^{\alpha T_{\min}} ||\bv|| & \text{ by \cite[Lemma 5.1]{shah1996limit}. }
	\end{align*}
\end{proof}
Let $D(\boldsymbol{\theta})=\min_{\alpha\in \Delta} \lambda_{\alpha}(\boldsymbol{\theta})$.
\begin{cor}\label{nondiv_1}
Let $\boldsymbol{\theta}$ be on the unit sphere in $\mC$ and $b_t=b^{\boldsymbol{\theta}}_{t}=\diag(e^{\theta_1t}, \cdots, e^{\theta_nt})$. Then, for any compact subset $L$ in $G/\Gamma$, there exists $\kappa=\kappa(n)>0$ and a $T_1=T_{1}(r,L,\boldsymbol{\theta})\gg_{r,L} D(\boldsymbol{\theta})^{-1}$ such that for every $0<\varepsilon<1$, any $z\in L$, and any $t \geq T_1$, 
$$
\nu_U(\{u\in B_U(r): b_tuz\notin K_\varepsilon\})\ll \varepsilon^{\kappa}\cdot \nu_U(B_U(r)).
$$
\end{cor}

\begin{proof}
By \autoref{nondiv_prop} applied to the irreducible representations of  $G$ on $\bigwedge^{j}(\R^{n})$,
$$
\sup \limits_{u \in B_{U}(r)} ||b_{t}ug\bv|| \geq c_{1}e^{\alpha t}||g\bv||.
$$
Then for any $g \in \pi^{-1}(L)$ and any $\bv \in \bigwedge^{j}(\Z^{n}) \backslash \{0\}$,
$$
\sup \limits_{u \in B_{U}(r)} ||b_{t}ug\bv|| \geq c_{2}e^{\alpha t},
$$
since $L$ is compact and $\bigwedge^{j}(\Z^{n})$ is discrete. 
Define $T$ to be such that $c_{2}e^{\alpha T} = 1$, and define $\phi(s) := b_{t} \Theta(s) g$, where $s \in \R^{d}$. Let $O$ be a neighborhood of $0$ in $\R^d$ such that $\Theta(O) = B_U(r)$. Then there is some $s \in O$ such that for any $t \geq T$, $||b_{t}\Theta(s)gv|| \geq 1$. Then by \autoref{KMnondiv_thm}, 
$$
\lambda(\{s\in O : b_t\Theta(s)z\notin K_\varepsilon\})\ll \varepsilon^{\frac{1}{d(n-1)}}\lambda(O).
$$
Then, since $\lambda$ and $\nu_U$ are absolutely continuous with respect to each other,
$$
\nu_U(\{u\in B_U(r) : b_tuz\notin K_\varepsilon\})\ll \varepsilon^{\frac{1}{d(n-1)}} \cdot \nu_U(B_U(r)).
$$
It now remains to show the dependence of $T=T_{1}$ on $r, L,$ and $\boldsymbol{\theta}$. Solving for $T$ yields $T=-\alpha^{-1}\log(c_{2})$. The constant $c_{2}$ depends on $r$ and $L$ and  $\alpha=c \min_{\alpha\in \Delta} \lambda_{\alpha}(\boldsymbol{\theta})=c D(\boldsymbol{\theta})$ for some $c>0$ depending on the choice of representation coming from \autoref{nondiv_prop}.
\end{proof}

\begin{cor}\label{nondiv2}
Let $\boldsymbol{\theta}$ be on the unit sphere in $\mC$ and $b_t=b_{t}^{\boldsymbol{\theta}}=\diag(e^{\theta_1t}, \cdots, e^{\theta_nt})$. Then there exists $\kappa=\kappa(n)>0$ and $T_2 = T_2(\boldsymbol{\theta})\gg D(\boldsymbol{\theta})^{-1}$, such that for any $0<\varepsilon<1$ and any $t\geq T_2$,
$$
\mu(\{z\in U.z_0: b_tz\notin K_\varepsilon\})\ll \varepsilon^{\kappa}.
$$
\end{cor}

\begin{proof}[Proof of \autoref{nondiv2}]
Since $U.z_0$ is periodic, $U\cap\Gamma$ is a uniform lattice in $U$. Then there is a relatively compact fundamental domain, $\Omega$, for $U/U\cap\Gamma$ in $U$. Cover each $u \in \Omega$ by $B_{U}(u,r)$ so that $\Omega \subseteq \bigcup_{u \in \Omega} B_{U}(u,r)$ and $\pi$ is injective on $B_{U}(u,r)$. Let $O(\log u, r)$ be a ball in $\lieu$ such that $\exp(O(\log u, r))=B_{U}(u,r)$. By the Besicovitch covering theorem, there exists a constant $c_d$, depending only on the dimension d, such that
$$
\Omega \subseteq \bigcup\limits_{i = 1}^{c_d} \bigcup\limits_{B \in \mathcal{O}_i} B, 
$$
where each $\mathcal{O}_i$ is a collection of disjoint balls $B_{U}(u,r)$. By \autoref{nondiv_1}, for $t\gg_{r} D(\boldsymbol{\theta})^{-1}$
$$
\nu_U(\{u\in B_{U}(u,r): b_tuz\notin K_\varepsilon\})\ll \varepsilon^{\frac{1}{dn}} \cdot \nu_U(B_{U}(u,r)).
$$
Then we have
\begin{align*}
	\nu_U(\{u \in \Omega : b_tuz \notin K_\varepsilon\}) & \leq  \nu_U(\{u \in \bigcup\limits_{i=1}^{c_d} \bigcup\limits_{B \in O_i}B: b_tuz \notin K_\varepsilon\})\\
	& \leq \sum\limits_{i=1}^{c_d} \sum\limits_{B \in O_i} \nu_U(\{u \in B: b_tuz \notin K_\varepsilon\})\\
	& \ll \sum\limits_{i=1}^{c_d} \sum\limits_{B \in O_i} \varepsilon^{\frac{1}{d(n-1)}} \cdot \nu_U(B)\\
	& = \sum\limits_{i=1}^{c_d} \varepsilon^{\frac{1}{d(n-1)}} \cdot \nu_U(\bigcup \limits_{B \in O_i}B)\\
	& \leq c_d \cdot \varepsilon^{\frac{1}{d(n-1)}} \cdot \nu_U(\Omega).
\end{align*}

Since $\Omega$ is a fundamental domain and $\mu$ is a probability measure,
$$
\mu(\{z\in U.z_0: b_tz\notin K_\varepsilon\})\ll \varepsilon^{\frac{1}{d(n-1)}}.
$$
\end{proof}

\subsection{Proof of \autoref{main1}}\label{pf1}

\begin{proof}
It suffices to prove the result for $a\in\mC$ since $\min_{\alpha\in\Delta} \lambda_{\alpha}(a)\leq1$ if $a\not\in \mC$ and the theorem (in the case $a\not\in \mC$) would follow by taking the implied constant to be a multiple of $\sup|\vp|$. So we may suppose $a\in\mC$. Write $a=g_{t}=\exp(t\boldsymbol{\theta})$ where $t>0,$ $\boldsymbol{\theta}\in\mC$ (reverting notation back to the Lie algebra) is on the unit sphere of $\mathfrak{a}$. Then the term appearing on the right hand side of (\ref{main1}) can be written as 
	\[
		\left( \min_{\alpha\in\Delta} \lambda_{\alpha}(a) \right)^{-\delta} = e^{-\delta t D}
	\]	
where $D=D(\boldsymbol{\theta})=\min_{\alpha\in \Delta} \lambda_{\alpha}(\boldsymbol{\theta})=\mathrm{dist}(\boldsymbol{\theta},\partial\mC)$, and where the $\lambda_\alpha$'s appearing on the left hand side are understood to be multiplicative (as they are in (\ref{main1})).\newline
\indent The proof follows the outline in \autoref{overview}. Notice that it suffices to prove that \autoref{equiTheoremMain} is valid whenever $t\gg 1/D$ because if $t\ll 1/D$, then $e^{-cDt}\gg e^{-c}$ and the left hand side of \autoref{equiTheoremMain} is trivially bounded by $2\sup|\vp|$. Therefore we assume that $t \gg D(\boldsymbol{\theta})^{-1}$ and that $\vp$ has mean zero. \newline 
\indent By \cite[Lemma 2.2]{KM} there exists a smooth function $\xi$ on $U$, whose support is contained in $B_U(r)$, satisfying $\xi\geq 0$, $\int_{U}\xi=1$, and $\|\xi\|_{\ell}\ll r^{-(k-\ell)}$. Note that a suitable $r$ will be chosen later. 

Write $g_t=a_tb_t$, where $a_t$ lies in the interior of the positive Weyl chamber $\mathcal{A}$ and $b_t$ lies in the interior of the convergence cone $\mathcal{C}$. Then
$$
\int_{U.z_0} \vp(g_t z) d\mu(z)=\int_{U.z_0} \vp(a_tb_t z) d\mu(z) =\int_{U.z_0}\int_{U} \xi(u)\vp(a_tub_t z)  d\nu_U(u)d\mu(z)
$$
where $z_{0}$ is the identity coset. To estimate the above integral, we partition $U.z_{0}$ as $U.z_{0}=B_1 \cup B_2$ and write
	\[
		\int_{U.z_0}\int_{U} \xi(u)\vp(a_tub_t z)  d\nu_U(u)d\mu(z)=I + II,
	\]
where
	\[
		I:=\int_{B_1}\int_{U} \xi(u)\vp(a_tub_t z)  d\nu_U(u)d\mu(z) \;\;\text{ and} \;\;\; II:= \int_{B_2}\int_{U} \xi(u)\vp(a_tub_t z)  d\nu_U(u)d\mu(z).
	\]
Let $\e=e^{-\beta t}$, where $\beta$ will be chosen later, and set
	\[
		B_1:=\lb  z\in U.z_{0} ~:~ b_{t}\cdot z\not\in K_{\e}  \rb.
	\]
With this choice for $B_1$ we have, by \autoref{nondiv2},
	\[
		| I | \leq  \sup|\vp|\mu(B_1) \int_{U} \xi(u)d\mu(u)\leq \sup|\vp| \e^{\kappa_1}.
	\]
To bound integral $II$, we define $B_2$ to be the complement of $B_1$ in $U.z_{0}$. Now we trivially have
	\[
		| II | \leq \int_{B_2}\left| \int_{U} \xi(u)\vp(a_tub_t z)  d\nu_U(u)\right| d\mu(z).
	\]
Since $a_{t}$ is in the interior of the positive Weyl chamber and $U$ is the full expanding horospherical subgroup of $a_{t}$, we may apply \autoref{equi}, with $f=\xi$ and $H=U$, to estimate the innermost integral. 

In order to satisfy the hypotheses of this proposition, we select $r$ so that $z\mapsto g.z$ is injective on $B_{G}(2r)B_{U}(r)\subset G$ for each $z\in B_2$. The injectivity radius of a set $L\subset X_{n}$ is defined by 
 	\[
		r(L):=\displaystyle\inf_{z\in L}\sup\lb  r>0 ~:~   z\mapsto g.z \text{ is injective on }B_{G}(r) \rb.
	\]
 By \cite[Proposition 3.5 ]{KM} the injectivity radius of $K_{\e}$ satisfies $r(K_{\e})\geq c \e^n$ for some $c>0$. It follows from the definition of the injectivity radius that $r(B_2)\geq r(K_\e)$ since $B_2\subset K_\e$.  Therefore, if we take $r= c \e^{n}/3=(c/3)e^{-n\beta t }$, then $z\mapsto g.z$ is injective on $B_{G}(2r)B_{U}(r)\subset G$ for each $z\in B_2$. Therefore, by \autoref{equi} and the assumptions on $\xi$, we have
 	\begin{eqnarray*}
		| II | 
		&\ll& \mu(B)\left( r\cdot\no[\vp]_{\rm Lip}\cdot\int_U|\xi|+ r^{-k}r^{-(k-\ell)} \cdot \no[\vp]_\ell \cdot e^{-\gamma \mathrm{dist}(a_{t},e)}\right) \\
		&\ll& (c/3)e^{-n\beta t }\no[\vp]_{\rm Lip}+  \cdot \no[\vp]_\ell \cdot e^{-\gamma \mathrm{dist}(a_{t},e)}e^{\beta t(2 k-\ell)n}.
	\end{eqnarray*}
There is a number $\eta=\eta(a_{1})>0$ such that  $e^{-\gamma \mathrm{dist}(a_{t},e)}\leq e^{-\gamma \eta t}$, and so we may write
	\[
		| I | + | II | \leq C \left(e^{-\beta \kappa t} +e^{-\beta t n} +e^{(\beta  (2k-\ell) n-\gamma \eta) t }\right)
	\] 
where $C$ depends only on $\vp$ and $k$. Note that $\kappa<1$ and so we can choose $\beta$ to equalize the exponents and we see that 
	\[
		\beta=\dfrac{\gamma \eta}{(2k-\ell)-\kappa}.
	\] 
The only term above which depends on the flow is $\eta.$ Recall that $D(\boldsymbol{\theta})=\min_{\alpha\in\Delta}\lambda_{\alpha}(\boldsymbol(\theta))=\lambda_{\beta}(\boldsymbol{\theta})$ is equal to $\mathrm{dist}(\boldsymbol{\theta},\partial\mC)$. Therefore we can choose $a_{1}\in\mA$ very close to a multiple of $\lambda_{\beta}$. That is, we can always choose the factorization $g_{t}=a_{t}b_{t}$ so that $a_{t}$ is close to $\R\lambda_{\beta}$ with magnitude approximately $\mathrm{dist}(\boldsymbol{\theta},\partial\mC)$. Therefore the factorization can be chosen so that $\eta=\tilde{c}\mathrm{dist}(\boldsymbol{\theta},\partial\mC)$ for some $\tilde{c}>0$ and the constant $\delta$ appearing in the statement of the theorem can be taken to be
	\[
		\delta=\dfrac{\tilde{c}\gamma}{(2k-\ell)-\kappa}>0.
	\]
Therefore  
	\[
		|I|+|II|\ll e^{-\delta t D }.
	\]
 This proves the result.
 \end{proof}
\subsection{Proof of \autoref{main2}}\label{pf2}
{\it To simplify notation and to keep this section brief, we will prove the case $E=\emptyset$ as the general case is similar.} Our proof of \autoref{main2} is basically an induction argument using \autoref{main1} as the base case. To describe the basic idea behind the proof let us first give an explicit description of the groups $Q_{F}$. A subset $F\subset \Delta$ can be described as
	\[
		F=\lb \alpha_{i_{1}},...,\alpha_{i_{\ell}} \rb\subset \Delta.
	\]
We will find it more convenient to work with the complement $\mE=\Delta\setminus F$ of $F$ in $\Delta$ rather than with $F$ itself. Finally we can describe $Q_{\mE}$ in matrix form as
	\[
		Q_{\mE}= 
			\left(
				\begin{array}{cccc}
				 \SL_{k_{1}}(\R) & \ast & \hdots & *\\
				  & \SL_{k_{2}}(\R) & &\vdots \\
				  & & \ddots & \ast\\
				  & & & \SL_{k_{\ell+1}}(\R)
				\end{array}
			\right)
	\]
where $k_{1}=i_{1}$, $k_{2}=i_{2}-i_{1}$,..., $k_{\ell}=i_{\ell}-i_{\ell-1}$, $k_{\ell+1}=n-(k_{1}+\cdots +k_{\ell})$. Notice that if $F=\emptyset$, then $Q_{\mE}=Q_{\Delta}=\SL_{n}(\R)$. \newline
Now we are in a position in which we can outline the basic idea behind the proof. By the same reasoning as in the beginning of the proof of \autoref{main1}, we may suppose that $a\in\mC$. Let $\nu_{\mE}=\nu_{Q_{\mE}}$ denote the Haar measure which is equal to $\mu_{\mE}$ in $G/\Gamma$ when restricted to a fundamental domain of $\Gamma$. Then $\nu_{\mE}$ decomposes into a product measure according to 
	\[
		\nu_{\mE}=\nu_{\SL_{k_{1}}(\R)}\otimes \cdots \otimes \nu_{\SL_{k_{\ell+1}}(\R)}\otimes \nu_{W_{\mE}}
	\]
where
	\[
		W_{\mE}=			\left(
				\begin{array}{cccc}
				 I_{k_{1}} & \ast & \hdots & *\\
				  & I_{k_{2}} & &\vdots \\
				  & & \ddots & \ast\\
				  & & & I_{k_{\ell+1}}
				\end{array}
			\right).
	\]	
The proof then proceeds by applying \autoref{main1} to each $\SL$ block on the diagonal. Of course we must deal with the translates of the factor $W_{\mE}$, but a Jacobian argument shows that the measure is invariant. \newline
To begin, we observe that the group $U$ can be written as 
	\[
				U=			\left(
				\begin{array}{cccc}
				 U_{k_{1}} & \ast & \hdots & *\\
				  & U_{k_{2}} & &\vdots \\
				  & & \ddots & \ast\\
				  & & & U_{k_{\ell+1}}
				\end{array}
			\right)
	\]
where  $U_{m}$ is the group of $m\times m$ unipotent upper triangular matrices. The Haar measure $\nu_{U}$ evidently admits the factorization
	\[
		\nu_{U}=\nu_{U_{k_{1}}}\otimes \cdots \otimes \nu_{U_{k_{\ell+1}}}\otimes \nu_{W_{\mE}}.
	\]
The corresponding factorization for $a=g_{t}=\exp(t\boldsymbol{\theta})$ is given by  
	\begin{equation}\label{gdecomp}
		a=g_{t}=g^{(1)}_{t}\otimes \cdots \otimes g^{(\ell+1)}_{t}
	\end{equation}
where $g^{(j)}_{t}$ is the corresponding block of length $k_{j}$ in $g_{t}$. In this way we see that 
 
	\[
		g_{t}\nu_{U}=g_{t}^{(1)}\nu_{U_{k_{1}}}\otimes \cdots \otimes g_{t}^{(\ell+1)}\nu_{U_{k_{\ell+1}}}\otimes g_{t}\nu_{W_{\mE}}.
	\]
To prove \autoref{main2} we will use the following two lemmas.
\begin{lem}\label{Winvar}
	If $g_{t}\in \mC_{\mE}$, then $g_{t}\nu_{W_{\mE}}=\nu_{W_{\mE}}$.
\end{lem}
\begin{proof}

We can write $ \nu_{W_{\mE}}=\nu_{M(k_{1}, n-k_{1})}\otimes \cdots \otimes \nu_{M(n-k_{\ell+1},  k_{\ell+1})}$ where $M(r,s)=M_{r\times s}(\R)$ is the space of $r\times s$ matrices, hence 
	\[
		g_{t}\nu_{W_{\mE}}=g_{t}^{(1)}\nu_{M(k_{1}, n-k_{1})}\otimes \cdots \otimes g_{t}^{(\ell+1)}\nu_{M(n-k_{\ell+1},  k_{\ell+1})}
	\]	
and $g_{t}^{(j)}\nu_{M(k_{j}, n-k_{1}-\cdots -k_{j-1})}=\mathrm{Jac}(g_{t})\nu_{M(k_{j}, n-k_{1}-\cdots -k_{j-1})}$. But $g_{t}^{(j)}=(c_{1}(t),...,c_{k_{j}}(t))$ acts by dilating the $i^{th}$ row by $c_{i}(t)$, and so $\mathrm{Jac}(g_{t}^{(j)})=\prod_{i=1}^{k_{j}}c_{i}(t)^{n-k_{1}-\cdots -k_{j-1}}=1$. So \[ g_{t}^{(j)}\nu_{M(k_{j}, n-k_{1}-\cdots -k_{j-1})}=\nu_{M(k_{j}, n-k_{1}-\cdots -k_{j-1})}\] and the lemma follows.
\end{proof}
\begin{lem}\label{productCon}
	If $\nu^{(1)}_{t},\nu^{(1)}_{t}$ are probability measures converging to $\nu^{(1)},\nu^{(2)}$ effectively as 
		\[
			|\nu^{(i)}_{t}(f_{i})- \nu^{(i)}(f_{i}) | \ll e^{-\gamma_{i}t},
		\]
	where the implied constant depends only on $\sup|f_{i}|$, $\| f_{i} \|_{C^{k}}$ and $\|f_{i}\|_{Lip}$, then the measure $\nu_{t}^{(1)}\otimes \nu_{t}^{(2)}$ converges to $\nu^{(1)}\otimes \nu^{(2)}$ 
	effectively as 
		\[
			|\nu^{(1)}_{t}\otimes \nu^{(2)}_{t}(F)- \nu^{(1)}\otimes \nu^{(2)}(F) | \ll \max_{i=1,2}\lb  e^{-\gamma_{i}t} \rb,
		\]
	where the implied constant may depend on $\sup|F|$, $\| F \|_{C^{k}}$, $\|F \|_{Lip}$, and the measure of the support of $F$.
\end{lem}
\begin{proof} This argument is a standard application of the triangle inequality. Observe
	\begin{eqnarray*}
		|\nu^{(1)}_{t}\otimes \nu^{(2)}_{t}(F)- \nu^{(1)}\otimes \nu^{(2)}(F) |
		&\leq & |\nu^{(1)}_{t}\otimes \nu^{(2)}_{t}(F)- \nu_{t}^{(1)}\otimes \nu^{(2)}(F)| \\ 
		&&+ |\nu^{(1)}_{t}\otimes \nu^{(2)}(F)- \nu^{(1)}\otimes \nu^{(2)}(F)| \\
		&\leq & \dint_{X_{1}} \left| \nu_{t}^{(2)}(F(x_1 , \cdot)) - \nu^{(2)}(F(x_1 , \cdot))\right| d\nu_{t}^{(1)}(x_1) \\
		&& +  \dint_{X_{2}} \left| \nu_{t}^{(1)}(F(\cdot, x_2)) - \nu^{(1)}(F(\cdot, x_2))\right| d\nu^{(2)}(x_2) \\
		&\ll& e^{-\gamma_{1}t} + e^{-\gamma_{2}t}. 
	\end{eqnarray*}
\end{proof}

Now we can finish off the proof of \autoref{main2}.
\begin{proof}[Proof of \autoref{main2}]
Let $f\in C^{\infty}_{comp}(G/\Gamma)$. Then by \autoref{Winvar}
	\[
		g_{t}\mu=g_{t}^{(1)}\nu_{U_{k_{1}}/\Gamma\cap U_{k_{1}}}\otimes \cdots \otimes g_{t}^{(\ell+1)}\nu_{U_{k_{\ell+1}}/\Gamma\cap U_{k_{\ell+1}}}\otimes \nu_{W_{\mE}/\Gamma}.
	\]
But by \autoref{main1} for each each $j$ there exists a constant $c=c_{n},D_{j}=D_{j}(\boldsymbol{\theta})>0$ such that 
	\[
		\left|  g_{t}^{(j)}\nu_{U_{k_{j}}/\Gamma\cap U_{k_{j}}}(f) - \nu_{\SL_{k_{j}}(\R)/\SL_{k_{j}}(\Z)}(f) \right|\ll  E(f)e^{-c D_{j}t} \ll \tilde{E}(f)e^{-c D_{j}t}
	\]
where $E(f)=\max\{ \| f\|_{\ell},\| f\|_{Lip} \}$ and $\tilde{E}(f)=\max\{m(\mathrm{\supp}(f)) \| f\|_{C^\ell},\| f\|_{Lip} \}$. By inductively applying \autoref{productCon} we obtain
	\[
		\left|  \nu_{W_{\mE}/\Gamma}\otimes \prod_{j} g_{t}^{(j)}\nu_{U_{k_{j}}/\Gamma}(f) -  \nu_{W_{\mE}/\Gamma}\otimes \prod_{j}\nu_{\SL_{k_{j}}(\R)/\SL_{k_{j}}(\Z)}(f) \right|\ll \tilde{E}(f)e^{-\min_{j} cD_{j}t}.
	\]
Therefore there exists a constant $c=c_{n}>0$ such that
	\[
		\left|  g_{t}\mu(f) - \nu_{\mE}(f)    \right| \ll e^{-c(\min_{j}D_{j})t}.
	\]
It remains to show that $\min_{j}D_{j}= \min_{\alpha\in \mE^c}\lambda_{\alpha}(\boldsymbol{\theta})$ (recall we are working with the complement of $F$ in $\Delta$).
 To see this, observe that with our choice of $\Delta$ the fundamental weights are given by
	\begin{equation}\label{fundamentalweight}
		\lambda_{\alpha_{j}}(\boldsymbol{\theta})=\theta_{1}+\cdots+\theta_{j}
	\end{equation}
and since $\boldsymbol{\theta}\in\mC_{\mE}$ 
	\[
		\lambda_{\alpha_{j}}(\boldsymbol{\theta})=0
	\]
if and only if $j=k_{i}$ for $i=1,...,\ell$. Recall the decomposition (\ref{gdecomp}) and notice that if $i_{s}<r<i_{s+1}$, then by (\ref{fundamentalweight})
	\begin{eqnarray*}
		\lambda_{\alpha_{r}}(\log(g_{t}))
		&=& t (\lambda_{\alpha_{i_{s}}}(\boldsymbol{\theta})+ (\theta_{i_{s}+1}+\cdots +\theta_{r}) )\\
		&=& t (\theta_{i_{s}+1}+\cdots +\theta_{r})  \\
		&=& \tilde{\lambda}_{\beta_{r}}(\log(g^{(s+1)}_{t}))
	\end{eqnarray*}
for some fundamental weight $\tilde{\lambda}_{\beta_{r}}$ of $\SL_{k_{s+1}}(\R)$. In particular
	\[
		D_{s}=\displaystyle\min_{i_{s}<r<i_{s+1}} \theta_{i_{s}+1}+\cdots +\theta_{r}
	\]
and, converting back to multiplicative notation (as at the end of the proof of \autoref{main1}) we have the desired result.
\end{proof}
\subsection{Proof of Theorem \ref{main3}}\label{proof_of_main3}
In the following proof we use the fact that for each $a\in \mC_{j}$, the horospherical subgroup corresponding to $a$ is 
		\[
			H_{j}= \lb 
					\left(\begin{array}{cc} 
						I & A \\ 0 & I
					\end{array} \right) ~:~ A\in M_{j\times (n-j)} 
				\rb.
		\]	
\begin{proof}
 Without loss of generality we may assume that $\vp$ has mean zero. By \cite[Lemma 2.2]{KM} there is a function $\xi\in C^{\infty}_{comp}(H_{j})$ that can be chosen so that $\|\xi\|_{\ell}\ll r^{-(k-\ell)}$ , $\xi\geq 0$, $\int_{H_{j}}\xi=1$, and the support contained in $B_{H_{j}}(r)$ where $r=e^{-\beta t}$ with $\beta$ is to be chosen later. 
Arguing as in the beginning of the proof of \autoref{main1} we may suppose that $a\in\mC$. Write $a=g_t=\exp(t\boldsymbol{\theta})=a_{t}b_{t}$ where $t>0$, $\boldsymbol{\theta}\in\mC_{j}$ is on the unit sphere of $\mathfrak{a}$, $t^{-1}\log a_{t}$ is a multiple of 
\[
	(1/j,...,1/j, -1/(n-j),...,-1/(n-j)),
\]
and $b_{t}\in\mathcal{C}_{j}$. Note that the action of $b_{t}$ on $H_{j}$ (the horospherical subgroup of $a_{t}$) is non-contracting.  
Clearly
$$
\int_{U}f(u) \vp(g_t uz) d\nu_U(u)
=\int_{H_{j}}\int_{U} \xi(h)f(u)\vp(a_tb_t uz) d\nu_U(u) d\nu_{H_{j}}(h)
$$
for each $z\in L$. Notice that the horospherical subgroup $H_{j}$ is contained in $U$. Using the change of variables $u\mapsto b_{-t}hb_tu$ and the left invariance of the measure $\nu_U$ we get that
$$
\int_{U}f(u) \vp(g_t uz) d\nu_U(u)
=
\int_{U}\int_{H_{j}} f(b_{-t}hb_tu)\xi(h)\vp(a_thb_tuz) d\nu_{H_{j}}(h)d\nu_U(u).
$$
Now we are in a position to estimate the above integral. To accomplish this break the integral into 2 pieces. We have 
\[
	\mathrm{dist}(e, b_{-t}hb_t)\leq e^{-\rho t}\mathrm{dist}(e,h)
\]
for any $h\in {H_{j}}$ where $\rho$ depends only on $g_{t}$. Notice the supports of the functions 
\[
		u\mapsto f_u(h)= f(b_{-t}hb_tu)
\]
are contained in $B=\supp(f)B_{U}(e^{-(\rho+\beta)t})$. Suppose $t>T_{1}>0$ is taken large enough so that $r=e^{-\beta t}<r_{0}/2$ and $\mu(B)\leq 2\mu(\supp(f))$, where $T_{1}$ is from \autoref{nondiv2} and $r_{0}$ is the injectivity radius of $L$. Let $\epsilon=(2/c)^{1/n}e^{-\beta t/n}$ and define
	\[
		\Omega= \lb u\in U ~:~ b_{t}uz\not\in K_{\epsilon}  \rb\;\;\; \text{ and } \;\;\; \Phi= B- \Omega.
	\]
Then we write
	\[
		\int_{U}\int_{H_{j}} f(b_{-t}hb_tu)\xi(h)\vp(a_thb_tuz) d\nu_{H_{j}}(h)d\nu_U(u)= I +II
	\]
where
	\[
		I=\int_{\Omega}\int_{H_{j}} f(b_{-t}hb_tu)\xi(h)\vp(a_thb_tuz) d\nu_H(h)d\nu_U(u) 
	\]
and
	\[
		II=\int_{\Phi}\int_{H_{j}} f(b_{-t}hb_tu)\xi(h)\vp(a_thb_tuz) d\nu_H(h)d\nu_U(u).
	\]
For $I$ we have by \autoref{nondiv2} and the assumptions on $t$ 
	\[
		\begin{split}
		|I|\leq \mu(\Omega) \sup|f|\sup|\vp|\int_{H_{j}}\xi(h)d\nu_{H_{j}}(h)
		&  \ll \epsilon^{\kappa_{2}} 2\mu(\supp(f)) \sup|f|\sup|\vp| \\
		& \ll_{\vp,f,n}e^{-\beta \kappa_{2}t/n}.
		\end{split}
	\]
For $II$ we have by \autoref{equi} to obtain
	\[
		|II|\leq \mu(\Phi) \left(  r\cdot\no[\vp]_{\rm Lip}\cdot\int_{H_{j}}|f_{u}|+ r^{-k}\cdot\no[f_{u}]_\ell \cdot \no[\vp]_\ell \cdot e^{-\gamma \mathrm{dist}(a_{t},e)}   \right).
	\] 
 As a compactly supported smooth function on $H$, the Sobelev norm of $f_u(h)$ is controlled by the norms of $f$ and $\xi$ because the conjugation by $b_{-t}$ on $H_{j}$ is non-expanding and hence the derivatives coming from $f$ do not increase. In particular (by \cite[Lemma 2.2]{KM})
 	\[
		\| f_{u} \|_{\ell} \ll_{f} r^{-(\ell +j(n-j)/2)}
	\] so we find that for some $\eta=\eta(\boldsymbol{\theta})>0$ we have 
 	\[
		|II|\ll_{f,\vp} e^{-\beta t} + e^{-(\gamma \eta - (k+\ell +j(n-j)/2)\beta) t}.
	\]
	Now we find that 
	\begin{eqnarray*}
		|I|+|II| &\ll_{f,\vp} & 
		e^{-\beta \kappa_{2}t/n}  + e^{-\beta t} + e^{-(\gamma \eta - (k+\ell +j(n-j)/2)\beta) t} \\
		 &\ll_{f,\vp} & e^{-\beta \min\lb \kappa_{2}/n , 1 \rb t}+e^{-(\gamma \eta - (k+\ell +j(n-j)/2)\beta) t}. 
	\end{eqnarray*}
	We now choose 
	\[
		\beta=\dfrac{\gamma\eta n}{ \min\lb \kappa_{2}/n , 1 \rb+k+\ell +j(n-j)/2}.
	\]
	To finish the proof we argue as in the end of the proof of \autoref{main1} to obtain the desired result.
	\end{proof}

\section{Counting Lifts of Horospeheres}\label{horospheres}

Let $d(\cdot,\cdot)$ be a right $G$-invariant and $K-$bi-invariant Riemannian metric on $X$ that is induced by the Killing form and satisfies
	\[
		d(x_{0},x_{0}a)=\| \log(a)\|
	\]
for each $a\in A$, where $\|\cdot\|$ is the Euclidean norm on $\mathfrak{a}$.  Let $G=KAU$ be the usual Iwasawa decomposition of $G$ and  let  $B(x,R)$ be the ball of radius $R$ centered at $x\in X$ with respect to $d$.

\subsection{Proof of \autoref{horo1thm}}
The proof of \autoref{horo1thm} is based on a counting argument that utilizes our main equidistribution theorem. To indicate the basic approach we must introduce a few technical facts. Firstly, it was proven in  \cite{MS} that $\overline{B}_{R}=\lb gU ~:~ d(x_{0},x_{0}g)\leq R \rb$ admits the decomposition
	\begin{equation}\label{modUdecomp}
		\overline{B}_{R}= KA_{R}U/U
	\end{equation}
where $A_{R}=A\cap \tilde{B}(R)$ and $\tilde{B}_{R}=\{  g\in G : d(x_{0},x_{0}g)\leq R \}$.

Let $\mathcal{U}=K\backslash Kg_{0}U$ and $\overline{\mathcal{U}}=\pi(\mathcal{U})$ where $\pi: X\rightarrow X/\Gamma$ is the natural projection. We observe that there exists an $a_{0}\in A$ such that  $Kg_{0}U=Ka_{0}U$. This is guaranteed by the Iwasawa decomposition. In \autoref{horo1thm} we are interested in the asymptotic behavior of the function
	\[
		N(R)=\#\lb \gamma\in\Gamma ~:~ \mathcal{U}\gamma\cap B(x_{0},R)\neq \emptyset  \rb.
	\]
The function 
	\[
		F_{R}(g\Gamma)=\dsum_{\gamma\in \Gamma/\Gamma\cap U}1_{R}(g\gamma a_{0} U),
	\]
where $1_{R}(x)$ is the characteristic function of $\overline{B}_{R}$, satisfies $F_{R}(\Gamma)=N(R)$.  Let 
	\[
		f(R)=\int_{A_{R}\cap \mC}\rho^{\prime}_{\Delta}(a)da,
	\]
where $\rho^{\prime}_{\Delta}(a)=\exp(\al \rho_{\Delta},\log(a) \ar)$ and $\rho_{\Delta}$ is the sum of the positive roots. We will also need the following proposition showing the existence of certain approximate identities, also known as mollifiers.

\begin{prop}\label{approxidentity}
 For each $\epsilon>0$ with $\epsilon\ll 1$ there exists a function $\Psi_{\epsilon}\in C^{\infty}_{comp}(G/\Gamma)$ such that (1) $\Psi_{\epsilon}$ is supported in a ball of radius $\epsilon$ centered at $\Gamma$, (2) $\Psi_{\epsilon} $ is non-negative with integral equal to 1, (3) $\sup \Psi_{\epsilon}\ll \epsilon^{-dim(G)}$, (4) $\|\Psi_{\epsilon}\|_{Lip}\ll \epsilon^{1-\dim(G)}$, and (5) for each $\ell\geq 1$ we have $\|\Psi_{\epsilon}\|_{\ell}\ll \epsilon^{-(\ell+\dim(G))}$, and  $\| \Psi_{\e}\|_{C^{\ell}}\ll \epsilon^{-(\ell+\dim(G))}$.
\end{prop}
\begin{proof}
 It suffices to prove the result in $\R^{N}$ for $\epsilon\ll 1$ where $N=\dim(G)$. Let $\Psi\in C^{\infty}_{comp}(\R^N)$ be non-negative with integral equal to 1 that is supported in the ball of radius 1 centered at the origin. Define
 	\[
		\Psi_{\epsilon}(\bx)=\epsilon^{-N}\Psi(\epsilon^{-1}\bx).
	\]
Then by construction $\Psi_{\epsilon}$ satisfies (1)-(3). Item (4) also follows from the definition: 
	\[
		\| \Psi_{\epsilon} \|_{Lip} =\sup_{\| \bx\|,\|\by\| <\epsilon} \dfrac{| \Psi_{\epsilon}(\bx) - \Psi_{\epsilon}(\by)|}{\| \bx -\by \|}
						=\sup_{\| \bx\|,\|\by\| <1} \dfrac{\epsilon^{-N}| \Psi(\bx) - \Psi(\by)|}{\epsilon^{-1}\| \bx -\by \|}=\epsilon^{1-N}\| \Psi\|_{Lip}.
	\]
Similarly item (5) follows from the observation that if $D=d/dx_{i_{1}}\cdots d/dx_{i_{\ell}}$, then $\| D \Psi_{\epsilon} \|_{L^p}= \epsilon^{-\ell}\| D \Psi \|_{L^p}$ when $0<p<\infty$, which follows by a change of variables, and $\| D \Psi_{\epsilon} \|_{L^{\infty}}= \epsilon^{-(N+\ell)}\| D \Psi \|_{L^\infty}$. Here $\| \cdot\|_{L^p}$ is the usual $L^p(\R^N)$ norm.
\end{proof}

Our proof of \autoref{horo1thm} is based on the following two lemmas. The first lemma, \autoref{weakConvergenceLemma1}, is a direct consequence of \autoref{main1}, and the second lemma, \autoref{approxConvergenceLemma1}, is essentially a corollary of \autoref{weakConvergenceLemma1}. We will prove the lemmas after we prove \autoref{horo1thm}. 
\begin{lem}\label{weakConvergenceLemma1}
There exists $\delta>0$ such that for every $\Psi\in C^{\infty}_{comp}(G/\Gamma)$
	\begin{equation}\label{weakConvergence1rate}
		\left| \al F_{R}, \Psi \ar  - \kappa f(R)\al 1,\Psi \ar   \right|
		\ll E(\Psi) e^{(\| \rho_{\Delta}\|-\delta)R},
	\end{equation}
where 
	\[
		\kappa=c_{n}\mathrm{vol}(K)\dfrac{\mathrm{vol}(\ol{\mU})}{\mathrm{vol}(\mM)}
	\]
for some constant $c_{n}>0$ that only depends on $n$ and $E(\Psi)=\max\{\|\Psi\|_{Lip},\| \Psi \|_{\ell}, \sup|\Psi|\}$.
\end{lem}

\begin{lem}\label{approxConvergenceLemma1}
 Suppose $\Psi_{\e}\in C^{\infty}_{c}(G/\Gamma)$ is the approximate identity supported in $\tilde{B}_{\e}$ given in \autoref{approxidentity}. If $\e=e^{-cR}$ for some $c>0$, then
 	\begin{equation}\label{approxConvergence1rate}
		\left|    F_{R}(\Gamma) - \al F_{R}, \Psi_{\e} \ar   \right|
		\ll R^{n-2}e^{(\| \rho_{\Delta}\| -c)R}+ E(\Psi_{\e})e^{(\|\rho_{\Delta} \|-\delta)R}.
	\end{equation} 
where $E(\Psi_{\e})=\max\{\|\Psi_{\e}\|_{Lip},\| \Psi_{\e} \|_{\ell}, \sup|\Psi_{\e}|\}$.
\end{lem}

\begin{proof}[Proof of \autoref{horo1thm}]
Assume \autoref{weakConvergenceLemma1} and \autoref{approxConvergenceLemma1}. Take $\Psi_{\epsilon}$ to be the approximate identity given by \autoref{approxidentity} supported in the ball of radius $\epsilon=e^{-cR}$ for some $c>0$ to be chosen later. Then it is immediate that
	\[
		\max\{\|\Psi_{\epsilon}\|_{Lip},\| \Psi_{\epsilon} \|_{\ell}, \sup|\Psi_{\epsilon}|\} \ll \epsilon^{-(N+\ell)}= e^{c(N+\ell)R}.
	\]
Then for some $u,\delta>0$
	\begin{eqnarray*}
		|N(R)-\mathrm{main}~\mathrm{term}|
		&=& |F_{R}(\Gamma)-\mathrm{main}~\mathrm{term}| \\ 
		&\leq&  |F_{R}(\Gamma)-\al F_{R},\Psi_{R} \ar |+|\al F_{R},\Psi_{R} \ar-\mathrm{main}~\mathrm{term}| \\
		&\ll& R^{u}e^{(\| \rho_{\Delta}\| - c)R} +   e^{c(N+\ell)R}e^{(\| \rho_{\Delta} \|-\delta)R}.
	\end{eqnarray*}
We now select $0<c<\delta(N+\ell+1)^{-1}$ to obtain 
	\[
		|N(R)-\mathrm{main}~\mathrm{term}| \ll e^{(\| \rho_{\Delta}\| - c)R}.
	\]
To finish the proof we observe that for each $s>0$ we have  $e^{(\| \rho_{\Delta}\|-s)R}\ll_{s} \mathrm{vol}(B(x_{0},R))$. In particular there is a number $0<q<1$ such that
	\[
		e^{(\| \rho_{\Delta}\|-c)R} \ll_{c} \left( \mathrm{vol}(B(x_{0},R)) \right)^{1-q} 
	\]
This can be seen by choosing $q$ such that $(\| \rho_{\Delta}\|-c)<\|\rho_{\Delta}\|(1-q)$. Plainly we may take any $0<q<c\|\rho_{\Delta}\|^{-1}$ and conclude that 
		\[
		|N(R)-\mathrm{main}~\mathrm{term}| \ll\left( \mathrm{vol}(B(x_{0},R)) \right)^{1-q} .
	\]
\end{proof}
\begin{proof}[Proof of \autoref{weakConvergenceLemma1}]
	Note that $\rho_{\Delta}$ is twice the sum of the positive root and so it is in the positive Weyl chamber. From the proof of Lemma 30 from \cite{MS} we find that 
		\[
			\al F_{R},\Psi \ar= \nu(\pi(U))\int_{K}\int_{A_{R}}\int_{G/\Gamma} \overline{\Psi(g^{\prime}\Gamma)}d(ka\mu_{U})(g^{\prime})\rho^{\prime}_{\Delta}(a)dadk
		\]
	where $A_{R}=A\cap \tilde{B}_{R}$, and $\rho_{\Delta}^{\prime}$ is given in \cite{MS}. For simplicity let us assume that $A_{R}$ is centered at the identity. Choose $\delta_{1}>0$ such that 
		\[
			Y=\left\lbrace  w \in\mathfrak{a} : \al w,\rho_{\Delta} \ar>(1-\delta_{1})\|\rho_{\Delta}\|\|w\|  \right\rbrace
		\]
	is contained in $\mathcal{A}$ and write 
		\[
			A_{R}=\Omega_{1}\cup \Omega_{2}
		\]
	where $\Omega_{1}=A_{R}\cap Y$ and $\Omega_{2}$ is the complement of $\Omega_{1}$ in $A_{R}$. We will decompose the Haar measure $da$ on $A$ in ``polar coordinates'' as 
		\begin{equation}\label{sphericalDecomp}
			da=d\exp(r\boldsymbol{\theta})=r^{n-2} dr d\sigma(\boldsymbol{\theta})
		\end{equation}
where $\boldsymbol{\theta}$ is an element of the unit sphere of $\mathfrak{a}$, $r>0$, $dr$ is a multiple of the Lebesgue measure on $\R$, and  $d\sigma$ is the measure on the unit sphere $S^{n-2}$ inherited from the Lebesgue measure. 
Observe that for each $k\in K$
		\begin{eqnarray*}
			 && \left| \int_{\Omega_{1}}\int_{G/\Gamma} \overline{\Psi(g^{\prime}\Gamma)}d(ka\mu_{U})(g^{\prime})\rho^{\prime}_{\Delta}(a)da - \int_{\Omega_{1}}\int_{G/\Gamma} \overline{\Psi(g^{\prime}\Gamma)}dm(g^{\prime})\rho^{\prime}_{\Delta}(a)da \right| \\
			&\leq &  \int_{\Omega_{1}} \left|\int_{G/\Gamma} \overline{\Psi(g^{\prime}\Gamma)}d(ka\mu_{U})(g^{\prime}) - \int_{G/\Gamma} \overline{\Psi(g^{\prime}\Gamma)}dm(g^{\prime}) \right| \rho^{\prime}_{\Delta}(a)da \\
			&=& \int_{\mathscr{S}}\dint_{0}^{R} \left|\int_{G/\Gamma} \overline{\Psi(g^{\prime}\Gamma)}d(k\exp(r\boldsymbol{\theta}))\mu_{U})(g^{\prime}) - \int_{G/\Gamma} \overline{\Psi(g^{\prime}\Gamma)}dm(g^{\prime}) \right| \rho^{\prime}_{\Delta}(\exp(r \boldsymbol{\theta}))d\sigma({\boldsymbol{\theta}}) \\
			&\ll& C\int_{\mathscr{S}}\dint_{0}^{R} e^{-cD(\boldsymbol{\theta})r}\rho^{\prime}_{\Delta}(\exp(r \boldsymbol{\theta}))r^{n-1}dr d\sigma({\boldsymbol{\theta}}) \\
			&\ll&   C  \int_{\mathscr{S}}\dint_{0}^{R} e^{-crD(\boldsymbol{\theta})}e^{r\al \rho_{\Delta},  \boldsymbol{\theta}\ar}r^{n-2}dr d\sigma({\boldsymbol{\theta}}) \\
			&\ll&   C  R^{n-1} e^{R(\| \rho_{\Delta}\|-\alpha_s)},
		\end{eqnarray*}
where we have applied \autoref{main1} on the third to last line and $C=C(\Psi,n)$ is the constant appearing in the statement of \autoref{main1}.  Here $\mathscr{S}$ is the intersection of the unit sphere and $\Omega_{1}$, and $\alpha_s=\min D(\boldsymbol{\theta})>0$ where the minimum is over all $\boldsymbol{\theta}\in \mathscr{S}$ and $D(\boldsymbol{\theta})=\mathrm{dist}(\boldsymbol{\theta},\partial\mC)$ as in the proof of \autoref{main1}.\newline

Replacing $\Omega_{1}$ with $\Omega_{2}$ above we have upon a trivial estimation 
	\begin{eqnarray*}
		\left| \int_{\Omega_{1}}  \right. \int_{G/\Gamma} &&  \overline{\Psi(g^{\prime}\Gamma)}d(ka\mu_{U})(g^{\prime})\rho^{\prime}_{\Delta}(a)da \\ 
			&-& \left. \int_{\Omega_{1}}\int_{G/\Gamma} \overline{\Psi(g^{\prime}\Gamma)}dm(g^{\prime})\rho^{\prime}_{\Delta}(a)da \right| \ll (\sup |\Psi| )R^{n-1}e^{R\| \rho_{\Delta} \|(1-\delta_{1}) }.
	\end{eqnarray*}

Following the proof of \autoref{main1} it is easily seen that $C\ll E( \Psi )$ where the implied constant depends only on $n$. Therefore we have
	\begin{eqnarray*}
    		\left|  \al F_{R},\Psi  \ar  - \vol(K)\nu(\pi(U))\int_{A_{R}a_{0}}\rho^{\prime}_{\Delta}(a)da\al 1,\Psi  \ar \right| \ll \mathrm{vol}(K) E(\Psi)R^{n-1}e^{R(\| \rho_{\Delta}\|-\delta_{2}) }
	\end{eqnarray*}
where $\delta_{2}=\min\lb \alpha_s,\| \rho_{\Delta}\|\delta_{1} \rb$>0.
\end{proof}


\begin{proof}[Proof of \autoref{approxConvergenceLemma1}]
Let $d=n-2$ be the dimension of $A_{R}$.  Observe that for each $g\in \tilde{B}_{\e}$ we have
	\begin{equation}
		F_{R-\e}(g\Gamma)\leq F_{R}(\Gamma) \leq F_{R+\epsilon}(g\Gamma),
	\end{equation}
which implies
	\begin{equation}\label{approx_ineq1}
		\al F_{R-\e},\Psi_{\e} \ar \leq F_{R}(\Gamma) \leq \al F_{R+\e},\Psi_{\e} \ar.
	\end{equation}
Then by \autoref{approx_ineq1}, we find that
	\begin{eqnarray*}
		\left|  F_{R}(\Gamma)-\al F_{R},\Psi_{\e} \ar   \right|
		& \leq & \al F_{R+\e}, \Psi_{\e} \ar-  \al F_{R-\e}, \Psi_{\e} \ar\\
		& \ll &  \int_{A_{R+\epsilon}-A_{R-\epsilon}} \rho^{\prime}_{\Delta}(a)da +E(\Psi_{\e})\mu_{A}(\psi_{S}(A_{R+\epsilon})) e^{(\|\rho_{\Delta} \| -\delta)(R+\epsilon)}
	\end{eqnarray*}
where $E(\Psi_{\e})=\max\{\sup|\Psi_{\e}|,\| \Psi_{\e}\|_{\ell},\|\Psi_{\e}\|_{Lip} \}$. Observe for $Z$ a small spherical shell about $\rho_{\Delta}$
	\[
		\int_{A_{R+\epsilon}-A_{R-\epsilon}} \rho^{\prime}_{\Delta}(a)da 
		= \int_{Z}\int_{R-\epsilon}^{R+\epsilon} e^{\al \rho_{\Delta}, \boldsymbol{\theta} \ar r}r^{n-1}dr d\sigma(\boldsymbol{\theta}) +O(e^{(\| \rho_{\Delta}\| -\delta ) R}).
	\]
By repeated applications of integration by parts we have for any $\omega\in\mathfrak{a}$
	\[
		\int_{Z}\int_{R-\epsilon}^{R+\epsilon} e^{\al \omega, \boldsymbol{\theta} \ar r}r^{d-1}dr d\sigma(\boldsymbol{\theta})
		=\int_{Z} \dsum_{p=0}^{d-1} (-1)^{p}  \left. c_{d,p}(\boldsymbol{\theta}) e^{\al \omega, \boldsymbol{\theta} \ar r}r^{d-1-p} \right|_{R-\epsilon}^{R+\epsilon} d\sigma(\boldsymbol{\theta})
	\]
where $c_{d,p}(\boldsymbol{\theta})>0$ for each $\boldsymbol{\theta}\in Z$.
Observe for real numbers $p,q$ we have
	\[
		(R+\e)^{p}e^{q(R+\e)}-(R-\e)^{p}e^{q(R-\e)} \ll R^{p}e^{qR}\sinh(q\e)
	\]

and $\sinh(q\e)=q\e + O(\e^{3})$ as $\e\ra0$ . To see this consider (after factoring out an $R^p$)
	\[
		\e=e^{-cR}\mapsto (1+\e/R)^{p}= \left( 1- \dfrac{c\e}{\log{\e}} \right)^{p}= 1 - \dfrac{cp\e}{\log(\e)}+o(c\e/\log(\e)).
	\]
Therefore
	\[
		(R+\e)^{p}e^{q(R+\e)}-(R-\e)^{p}e^{q(R-\e)}=R^{p}e^{q R}\left(  e^{q\e} - e^{-q \e} +O_{c,p}(\e/\log(\e))  \right).
	\]
It then follows (take $q=\|\rho_{\Delta}\|$) that 
	\begin{eqnarray*}
		\left|  F_{R}(\Gamma)-\al F_{R},\Psi_{\e} \ar   \right| 
		&\ll& R^{d-1}\epsilon  e^{\| \rho_{\Delta}\| R}+E(\Psi_{\e})e^{(\| \rho_{\Delta}\|-\delta)(R+\e)} \\
		&=& R^{d-1} e^{(\|\rho_{\Delta}\| -c)R}+ E(\Psi_{\e})e^{(\| \rho_{\Delta}\|-\delta)(R+\e)}. \qedhere
	\end{eqnarray*}
\end{proof}

\section{Proof of \autoref{maninThm}}\label{rational}

In this section we consider the problem of counting the number of rational points on a flag variety with respect to the anticanonical line bundle and prove \autoref{maninThm}. \newline
 \indent Let $X=G/P_{E}$ where $P_{E}$ is a standard parabolic subgroup of $G$ determined by $E$. Since $\rho_{E}\in \mA$, it follows that there is a unique finite dimensional irreducible representation $\eta:G\ra GL(V)$ for which $\rho_{E}$ is the highest weight. Moreover, there exists a $v_{0}\in V(\Q)$ such that  
	\[
		P_{E}=\lb  g\in G ~:~ \eta(g)[v_{0}]=[v_{0}] \rb
	\]
and $X$ is homeomorphic to the orbit $\eta(G)[v_{0}]$.  We will now define the height on $X$ with respect to $L$.

Let $H:\mathbb{P}(V)(\Q)\ra \R^+$ be defined by $H([v])=\| v \|$ where $[v]$ is the point in projective space corresponding to $v\in V$ corresponding to a primitive $v$ and $\| \cdot \|$ is the Euclidean norm on $V$. Now the height function with respect to the anticanonical bundle $\mathcal{L}$ is then
	\[
		h(x)=H(\eta(g_{x})[v_{0}])
	\]
where $g_{x}\in G$ is the unique point for which $\eta(g_{x})[v_{0}]=x$. We wish to determine the asymptotic of the function 
	\[
		N(T)=\# \lb x\in X(\Q) ~:~ h(x)\leq T     \rb.
	\]
We will not, however, deal directly with this function. By a theorem of Borel and Harish-Chandra, $(G/P_{E})(\Q)$ can be written as a finite union of $\Gamma$-orbits. This reduces the problem to studying a single $\Gamma$ orbit. Therefore we study
	\[
		N_{T}=\# \lb \gamma\in \Gamma / \Gamma \cap P_{E}  ~:~ \| \eta(\gamma)v \| <T \rb,
	\]
for $v\in V$ having $\| v\|=1$. Notice that 
	\[
		F_{T}(g\Gamma)= \dsum_{\gamma \in \Gamma/\Gamma\cap P_{E}} 1_{T}(\eta(g\gamma)v )
	\]
is equal to $N_{T}$ when $g=e$ and $1_{T}(\cdot)$ is the characteristic function of $B_{T}=\lb v\in V ~:~ \| v \| <T \rb$. Let $\tilde{B}_{T}$ be the corresponding subset of $G$, i.e.
	\[
		\tilde{B}_{T}=\lb g\in G ~:~  \eta(g\gamma)v \in B_{T} \rb.
	\]
Let $\ol{B}_{T}$ be the image of $\tilde{B}_{T}$ in $G/Q_{E}$. If $F\subset \Delta$ and  $a\in A$, then the $F$-projection $a_{F}$ of $a$ (defined in \cite[\S 4]{MS}) is the unique element $a_{F}\in A$ such that $\lambda_{\alpha}(a_{F})=\lambda_{\alpha}(a)$ for each $\alpha\in F$ and  $\lambda_{\alpha}(a_{F})=1$ for each $\alpha\not\in F$. By \cite[Lemma 32]{MS} $\ol{B}_{T}$ can be decomposed as 
	\[
		\tilde{B}_{T}=KA_{E^c, T} Q_{E} /Q_{E}
	\]
where $A_{E^c, T}=\lb  a\in A ~:~   a=a_{E^c}, ~\rho_{E}(a)\leq T \rb$. Let 
	\[
		f(T)=\dint_{A^{+}_{E^c,T}} \rho^{\prime}_{E }(a)da
	\]
where $A_{E^c, T}^{+}=A_{E^c, T}\cap \mC$ and $\rho^{\prime}_{E }$ is a character of $P_{E}$ given by 
	\[
		(\wedge^{\dim R_{u}P_{E}}\mathrm{Ad})(p)u=\rho^{\prime}_{E }(p)u
	\] 
for any $u\in \wedge^{\dim R_{u}P_{E}}Lie(R_{u}(P_{E}))$ where $R_{u}P_{E}$ is the unipotent radical of $P_{E}$. We will need the fact  (see \cite[\S 3]{MS}) that there is a vector $\rho_{E}\in \mathfrak{a}$ (in the logarithm of the convergence cone) such that 
	\[
		\rho_{E}^{\prime}(a)=\exp(\al \rho_{E},\log(a) \ar)
	\]
for any $a\in A$.
\begin{lem}
	Let $\Psi\in  C^{\infty}_{comp}(G/\Gamma)$. Then there exist constants $C,r,\delta>0$ such that 
	\[
		\left|   \al F_{T}, \Psi \ar   -  Cf(T) \int_{\G/\Gamma} \Psi(g)dg  \right|
		\ll   E(\Psi)T e^{-\delta\sqrt{\log{T}}}+  T\log(T)^{-r}\sup|\Psi|,
	\]
	where $\tilde{E}(\Psi)=\max\{\supp(\Psi)\| \Psi\|_{C^\ell},\|\Psi\|_{Lip},\sup|\Psi| \}$.
\end{lem}
\begin{proof}
Let $d=\dim(\mC_{E^c})$. Suppose $\Psi$ is supported in the ball of radius $\epsilon$. If $\lambda_{\alpha}(a)<e^{-\epsilon}$, then $\int_{G/\Gamma} \Psi (kad\mu_{Q_{E}})=0$. By unfolding the $F_{T}$ in the integral we obtain
	\[
		\al F_{T}, \Psi \ar = \dint_{K}\dint_{A^{(\e)}_{E^c, T}}\dint_{G/\Gamma} \Psi(g\Gamma) d(ka\mu_{Q_{E}})(g) \rho^{\prime}_{E}(a)dadk
	\]
where $A^{(\e)}_{E^c, T}=\lb a\in A_{E^c, T} ~:~ \lambda_{\alpha}(a)>e^{-\e}, \text{ for each }\alpha\in E^c \rb$. We estimate the integral by splitting $A^{(\e)}_{E^c, T}$ into two disjoint pieces, $\Omega_{T}^{(1)}$ and $\Omega_{T}^{(2)}$. Define 
	\[
		\Omega_{T}^{(1)}= \lb a\in  A^{+}_{E^c, T} ~:~ \mathrm{dist}\left( \dfrac{\log(a)}{\| \log(a) \|}, \partial \mC_{E^c}\right)>\sqrt{\dfrac{1}{\log(T)}} \rb
	\]
where $A^{+}_{E^c, T}=\lb  a\in A_{E^c,T} ~:~  \lambda_{\alpha}(a)\geq 1 , \forall \alpha \rb$. 
Then on $\Omega_{T}^{(1)}$ we have, by \autoref{main2}, 
	\begin{eqnarray*}
		&& \dint_{K}\dint_{\Omega^{(1)}_{T}}\dint_{G/\Gamma} \Psi(g\Gamma) ((ka)d\mu_{Q_{E}})(g) \rho^{\prime}_{E}(a)dadk \\
		&=&  \dint_{K}\dint_{S^{(1)}}\dint_{0}^{\log(T)/\al \rho_{E},\theta \ar}\dint_{G/\Gamma} \Psi(g\Gamma) ((k\exp(R\theta))d\mu_{Q_{E}})(g) e^{\al \rho^{\prime}_{E}, \theta \ar R}R^{n-2}dRd\sigma(\theta)dk \\
		&=&\mathrm{vol}(K) \dint_{\Omega^{(1)}_{T}}\rho^{\prime}_{E}(a)da \dint_{G/\Gamma} \Psi(g\Gamma) dg + O\left(\tilde{E}(\Psi)Te^{-\delta\sqrt{\log(T)}}\right) \\
	\end{eqnarray*}
where $S^{(1)}$ is the intersection of $\Omega^{(1)}_{T}$ with the unit sphere in $\mathfrak{a}$ and we have used the dependence on the implied constant on $\Psi$ in \autoref{main2} by following its proof. Now we estimate the integral on $\Omega^{(2)}_{T}$ trivially to obtain
\[
	\left| \dint_{K}\dint_{\Omega^{(2)}_{T}}\dint_{G/\Gamma} \Psi(g\Gamma) ((ka)d\mu_{Q_{E}})(g) \rho^{\prime}_{E}(a)dadk -  \mathrm{vol}(K)\mathrm{vol}(\Omega_{T}^{(2)})\dint_{G/\Gamma}\Psi \right|
		\ll\mathrm{vol}(\Omega_{T}^{(2)})\sup|\Psi|. 
\]
But $\mathrm{vol}(\Omega_{T}^{(2)}) \ll Tp(\log(T))\log(T)^{-r}$ for some $r>0$ depending only on the dimension. 
\end{proof}

\begin{prop}\label{maninWeak}
	Let $d$ be the dimension of  $A_{E^c,T}$ defined above. Then there exists a polynomial $p(s)$ of degree $d-1$ such that 
		\begin{eqnarray*}
		 f(T)= \int_{A_{E^c, T}} \rho^{\prime}_{E}(a)da =Tp(\log T).
		\end{eqnarray*}
\end{prop}
%

\begin{lem}\label{maninApprox}
	Suppose $\Psi_{\e}\in C^{\infty}_{c}(G/\Gamma)$ is supported in the ball of radius $\e>0$ about $\Gamma$ and that $B_{T+\e}\subset \supp(\Psi_{\e})B_{T}$ and $\supp(\Psi_{\e})B_{T-\e}\subset B_{T}$. Then
		\[
			F_{T-\e}(g)\leq F_{T}(e) \leq F_{T+\e}(g)
		\]
	for each $g\in \supp(\Psi_{\e})$, and if $\e=\log(T)^{-c}$, then
		\[
			\left|  \al F_{T}, \Psi_{\e}\ar -F_{T}(e)  \right| \ll \tilde{E}(\Psi_{\e})Te^{-\delta\sqrt{\log(T)}}+ \log(T)^{d-1-c}+(\sup|\Psi_{\e}|) Tp(\log(T))\log(T)^{ -r },
		\]
where $r>0$ is the exponent coming from the previous lemma.
\end{lem}
\begin{proof}
	Observe that 
		\begin{eqnarray*}
			\left|  \al F_{T}, \Psi_{\e}\ar -F_{T}(e)  \right| 
			&\leq&  \al F_{T+\e}, \Psi_{\e}\ar -\al F_{T-\e}, \Psi_{\e}\ar \\
			&\ll & f(T+\e)-f(T-\e) \\ && + \tilde{E}(\Psi_{\e}) Te^{-\delta\sqrt{\log(T)}}+(\sup|\Psi_{\e}| )Tp(\log(T))\log(T)^{-r}.
		\end{eqnarray*}
	But $f(T)=Tp(\log(T))$ for some polynomial $p(x)$ of degree $d-1$, so $f(T+\e)-f(T-\e)= 2\e f^{\prime}(T)+o(\e)=2(p(\log(T))+p^{\prime}(\log(T)))\log(T)^{-c}+o(\log(T)^{-c}).$ \end{proof}

\begin{proof}[Proof of \autoref{maninThm}]
	We let $\Psi_{\e}$ be the approximate identity given in \autoref{approxidentity} that is supported in the ball of radius $\e=\log(T)^{-c}>0$ about $\Gamma$ (for some $c>0$ to be chosen later) and that $B_{T+\e}\subset \supp(\Psi_{\e})B_{T}$ and $\supp(\Psi_{\e})B_{T-\e}\subset B_{T}$.   Then by \autoref{approxidentity}, \autoref{maninApprox} and \autoref{maninWeak} we have for some $s,\delta>0$
	\begin{eqnarray*}
		|N(T)-\mathrm{main}~\mathrm{term}|
		&=& |F_{T}(\Gamma)-\mathrm{main}~\mathrm{term}| \\ 
		&\leq&  |F_{T}(\Gamma)-\al F_{T},\Psi_{\e} \ar |+|\al F_{T},\Psi_{\e} \ar-\mathrm{main}~\mathrm{term}| \\
		&\ll&    \tilde{E}(\Psi_{\e})Te^{-\delta\sqrt{\log(T)}}+  \log(T)^{d-1-c}+(\sup|\Psi_{\e}|) Tp(\log(T))\log(T)^{ -r }\\
		&\ll&  Tp(\log(T))\log(T)^{cs -r }.
	\end{eqnarray*}
We choose $c$ such that $0 < c < r/s$. The remainder of the proof regarding the volume estimate is similar to the end of the proof of \autoref{horo1thm}.
\end{proof}

\noindent {\bf Acknowledgements} We wish to  sincerely thank the anonymous referees for pointing out mistakes and inconsistencies in the preliminary versions of this paper and for many other helpful suggestions.  Many thanks go to our editor Dmitry Kleinbock for his care, his insights, and his attention to detail. Finally, we most sincerely thank Amir Mohammadi for suggesting this problem and generously sharing his ideas  with us. During the course of this work Michael Kelly was partially supported by NSF grants DMS-1045119, DMS-0943832, and DMS-1101326; and Han Li was partially supported by an AMS Simons Travel Grant.

\bibliographystyle{plain}
\bibliography{expanding}

\end{document}